\documentclass[oneside,11pt,doubleblind]{article}

\usepackage{geometry}
\usepackage{censor}

\usepackage[displaymath, mathlines]{lineno}
 \usepackage[onehalfspacing]{setspace}

\usepackage{graphicx}

\usepackage[english]{babel}	
\usepackage[utf8]{inputenc}
\usepackage{mathtools}
\usepackage{amssymb}
\usepackage{amsthm}

\usepackage{amsmath}
\usepackage{amsfonts}
\usepackage{dsfont}
\usepackage{tikz-cd}
\usepackage{xcolor}
\usepackage{graphicx}
\usepackage{hyperref}
\usepackage{enumitem}
\usepackage{pgfplots}
\usepackage{booktabs}

\usepackage{authblk}
\newtheorem{theorem}{Theorem}[section]
\newtheorem{lemma}[theorem]{Lemma}

\newtheorem{proposition}[theorem]{Proposition}				
\newtheorem{definition}[theorem]{Definition}
\newtheorem{example}[theorem]{Example}		
\newtheorem{remark}[theorem]{Remark}
\newtheorem{assumption}[theorem]{Assumption}
\newtheorem{algorithm}[theorem]{Algorithm}

\newcommand{\rot}{}
\newcommand{\inner}[1]{\langle#1\rangle}       % Skalarprodukt
\newcommand{\abs}[1]{\lvert #1 \rvert}

\newcommand{\C}{\mathds C}

\newcommand{\N}{\mathds N}

\newcommand{\R}{\mathds R}

\newcommand{\X}{\mathds X}
\newcommand{\Y}{\mathds Y}
\newcommand{\Hr}{\mathds H}
\newcommand{\Z}{\mathds Z}

\DeclareMathOperator{\sign}{sign}

\DeclareMathOperator{\ran}{ran}
\DeclareMathOperator{\dom}{dom}
\DeclareMathOperator{\id}{Id}

\DeclareMathOperator{\supp}{supp}

\newcommand{\Fo}{\mathcal{F}}
\newcommand{\Ko}{\mathcal{K}}
\newcommand{\Uo}{\mathcal{U}}
\newcommand{\Vo}{\mathcal{V}}
\newcommand{\Wo}{\mathcal{W}}

\newcommand{\Io}{\mathcal{I}}

\newcommand*\diff{\mathop{}\!\mathrm{d}}

\DeclarePairedDelimiter\kl{(}{)}

\DeclarePairedDelimiter\norm{\lVert}{\rVert}

\newcommand{\la}{\lambda}
\newcommand{\La}{\Lambda}
\newcommand{\om}{\omega}

\newcommand{\al}{\alpha}

\newcommand{\qsv}{\kappa}

\newcommand{\tik}{\mathrm{T}}
\newcommand{\regfun}{\mathrm{R}}
\newcommand{\data}{g}
\newcommand{\signal}{f}

\newcommand{\reg}{\mathcal{R}}
\newcommand{\wvd}{\mathcal{D}}
\newcommand{\legendre}{\mathcal{L}}

\newcommand{\filt}{\Phi}
\newcommand{\soft}{\operatorname{soft}}

\numberwithin{equation}{section}

\allowdisplaybreaks

\title{Translation invariant diagonal frame decomposition of inverse problems and their regularization}

\date{February 22, 2023}

\author{Simon G\"oppel}
\affil{Department of Mathematics, University of Innsbruck\authorcr
 Technikerstrasse 13, 6020 Innsbruck, Austria
 \authorcr E-mail:  \texttt{simon.goeppel@uibk.ac.at}}

\author{J\"urgen Frikel}
\affil{Department of Computer Science and Mathematics\authorcr
Universit\"atsstra{\ss}e 31, D-93053 Regensburg, Germany
\authorcr E-mail: \texttt{juergen.frikel@oth-regensburg.de}
}

\author{Markus Haltmeier}
\affil{Department of Mathematics, University of Innsbruck\authorcr
Technikerstrasse 13, 6020 Innsbruck, Austria
 \authorcr E-mail:  \texttt{markus.haltmeier@uibk.ac.at}
 }

\begin{document}
\maketitle

\begin{abstract}
Solving inverse problems is central to a variety of important applications, such as biomedical image reconstruction and non-destructive testing.  These problems are characterized by the sensitivity of direct solution methods with respect to data perturbations. To stabilize the reconstruction process, regularization methods have to be employed. Well-known regularization methods are based on frame expansions, such as the wavelet-vaguelette (WVD) decomposition, which are well adapted to the underlying signal class and the forward model and furthermore allow efficient implementation. However, it is well known that the lack of translational invariance of wavelets and related systems leads to specific artifacts in the reconstruction.  To overcome this problem, in this paper we introduce and analyze the translation invariant diagonal frame decomposition (TI-DFD) of linear operators as a novel concept generalizing the SVD. We characterize ill-posedness via the TI-DFD and prove that a TI-DFD combined with a regularizing filter leads to a convergent regularization method with optimal convergence rates. As illustrative example, we construct a wavelet-based TI-DFD for one-dimensional integration, where we also investigate our approach numerically. The results indicate that filtered TI-DFDs eliminate the typical wavelet artifacts when using standard wavelets and provide a fast, accurate, and stable solution scheme for inverse problems.

\end{abstract}

\section{Introduction}

In this paper, we study with the stable solution of linear inverse problems. Such problems aim at recovering an unknown image or function $\signal_\star \in L^2(\R^d)$ from data
\begin{linenomath*}
\begin{equation}\label{eq:ip}
	\data^\delta = \Ko \signal_\star + \eta \,.
\end{equation}
\end{linenomath*}
Here  $\Ko \colon \dom (\Ko) \subseteq L^2(\R^d) \to \Y$ is a closed not necessarily bounded, linear operator between $L^2(\R^d)$ and another Hilbert space $\Y$ and $\eta$ denotes the unknown data distortion, which in our analysis is assumed to satisfy  $\norm{\eta} \leq \delta$ for the noise level $\delta >0$.  Inverse problems of the form \eqref{eq:ip} are often ill-posed, meaning that the solution is either not well-defined  or unstable with respect to data perturbations. In order to stabilize the solution  process one has to apply regularization methods \cite{engl1996regularization,markusvariation}. The basic idea of regularization is to incorporate prior information and to relax the exact solution concept to make the inversion process stable.

A common and successful approach for solving inverse problems is variational regularization. In this context, prior information is incorporated in the form  of a regularizer $\regfun \colon L^2(\R^d) \to [0, \infty]$ acting as penalty and  regularized solutions are constructed as minimizers of the generalized Tikhonov functional
\begin{linenomath*}
\begin{equation} \label{eq:tikhonov}
 \tik_{\al, \data^\delta}(\signal) \coloneqq \norm{\Ko \signal - \data^\delta}^2 + \alpha \regfun(\signal) \,.
 \end{equation}
\end{linenomath*}
A particularly important and well-studied special case is classical Tikhonov regularization where $\regfun(\signal) = \norm{\signal}^2$. In this case, assuming a singular value decomposition (SVD) $(u_\la, v_\la, \sigma_\la)_{\la \in \La}$ of the operator $\Ko$, the minimizer of the Tikhonov functional can be explicitly computed as $\signal_\alpha^\delta  = \sum_{\la \in \La}   ( \sigma_\la /( \sigma_\la^2 +\alpha)  ) \cdot  \inner{\data^\delta, v_\la} \;  u_\la $.
Opposed to general variational regularization, where the Tikhonov functional  \eqref{eq:tikhonov} has to be minimized via iterative  methods, the SVD based reconstruction can be evaluated directly and efficiently. The same holds true for more general filter based methods where the Tikhonov filter function $\filt_\al(\sigma) = \sigma / ( \sigma^2 +\alpha )$ is replaced by general functions.  Another  typical choice in that context  is the truncation filter $ \filt_\al(\sigma) =  \chi_{[\al, \infty)}(\sigma^2)$ which leads to the also well-known truncated SVD reconstruction  method.

\subsection{Frame-based regularization}

Although the SVD is a very successful tool to design stable inversion schemes for \eqref{eq:ip}, it comes with several shortcomings. Firstly, given a particular linear operator, the SVD may not be known analytically and  hard to compute numerically. Furthermore, the underlying orthonormal basis elements $u_n, v_n$ of the singular system are derived only from the operator $\Ko$ and are not adapted to signal classes of interest. On the other hand it is well known that different  function systems such as wavelet frames  are better adapted to functions classes of practical relevance. Therefore researchers studied frame based regularization methods where the regularizer in \eqref{eq:tikhonov} has the form  $\regfun(\signal) = \sum_{\la \in \La} r_\la(\inner{u_\la, \signal} ) $ for some frame or basis  $(u_\la)_{\la \in \La}$  of $L^2(\R^d)$ and univariate functionals $r_\la$. Such methods have been successful applied and analyzed in various settings  \cite{Daubechies_2004,dicken1996wavelet,Grasmair_2008,hubmer2020atmospheric, hubmer2021tomography, Lorenz_2008,Ramlau_2006,rieder1997wavelet,markusvariation}.

Variational frame based methods have the drawback that they again require computationally costly iterative solution methods. To overcome this issue, diagonal frame decompositions (DFD) have been developed as an alternative tool.  A DFD for an operator $\Ko$ is given as a family $(u_\la, v_\la, \qsv_\la)_{\la \in \La}$, where $(u_\la)_{\la\in \La}$ and  $(v_\la)_{\la\in \La}$ are chosen as frames instead of orthonormal bases as it is in the case of an SVD. Additionally, the $\qsv_\la$ are generalized singular values or so-called quasi singular values which satisfy $\Ko^* v_\la = \qsv_\la u_\la$.
As proposed and analyzed in \cite{ebner2020regularization} such DFDs yield explicit regularization methods  in the form of filtered DFDs
\begin{linenomath*}
\begin{equation} \label{eq:svd-filt}
	\signal_\alpha^\delta  =
	\sum_{\la \in \La} \filt_\al(\qsv_\la)   \,  \inner{\data^\delta, v_\la}  w_\la \,,
\end{equation}
\end{linenomath*}
where  $(\filt_\al)_{\al >0}$ is a regularizing filter and  $(w_\la)_{\la\in \La}$ a dual frame  to $(u_\la)_{\la\in \La}$. Such filtered DFDs combine advantages of filter based and frame based methods.
It may be seen as a generalization of the filtered SVD using redundant frames instead orthogonal  bases.  Precise definitions  of DFDs and  regularizing filters are provided in Definitions \ref{def:fdc} and \ref{def:regfilt} below. DFDs also generalize the wavelet vaguelette decomposition (WVD) introduced in the seminal paper
 \cite{donoho1995nonlinear}; for related constructions using curvelet and shearlet frames  see \cite{candes2002, COLONNA2010232}. DFDs for general frames has bee first  introduced in  \cite{frikel2020sparse} where they are analyzed in combination with soft thresholding. A convergence analysis of filtered DFDs for general linear filters has first been provided in   \cite{ebner2020regularization} and later in  \cite{hubmer2022regularization}.

\subsection{Proposed translation invariant DFD}

Classical frame based reconstruction approaches such as the WVD  lack translation invariance of the underlying frame which introduces well-known  wavelet  artifacts in the reconstruction.  Translation invariant wavelet frames are known  to perform better  in that regard for simple tasks such as denoising  \cite{coifman1995translation,nason1995stationary}.  Thus, it is natural to introduce an extension of the concepts to translation invariant frames for general inverse problems as we will do in the present paper. Furthermore, we present a complete convergence analysis for the  corresponding  filtered translation invariant frame decomposition (TI-DFD).

A translation  invariant  frame of $L^2(\R^d)$, which  is not a frame  in the classical sense, consists  of a family  $(u_\la)_{\la \in \La} \in L^2(\R^d)^\Lambda$ such that any $\signal  \in L^2(\R^d)$ can be stably analyzed and reconstructed via convolutions with $u_\la$.
In this paper we introduce  the TI-DFD of a linear operator as as family $(u_\la, \Vo^*_\la, \qsv_\la)_{\la \in \La}$ where $(u_\la)_{\la \in \La} \in L^2(\R^d) $ is TI-frame, $\Vo^*_\la \colon  \Y \to L^2(\R^d)$ are suitable linear operators and $\qsv_\la >0$ are generalized singular values  with  $ \Vo^*_\la (\Ko \signal) = \qsv_\la  \cdot (u_\la^* \ast \signal)$, and $u_\la^*(x)$ denotes the complex conjugate of $u_\la(-x)$. Given a regularizing  filter $(\filt_\al)_{\al>0}$ and a dual TI-frame $(w_\la)_{\la \in \La} \in L^2(\R^d)$ we will demonstrate that
\begin{linenomath*}
\begin{equation} \label{eq:svd-filt}
	\signal_\alpha^\delta  =
	\sum_{\la \in \La} \filt_\al(\qsv_\la)   \,  ( w_\la \ast (\Vo^*_\la \data^\delta) )
\end{equation}
\end{linenomath*}
is a regularization method and we additionally derive convergence  rates.  Precise definitions of TI frames  and TI-DFDs  are given  in Definitions~\ref{def:ti-frame}, \ref{def:ti-dfd} below.  For illustrative purpose, we will construct an example  of a TI-DFD for the 1D integration operator using  the TI wavelet transform. We will demonstrate that the resulting filtered TI-WVD  works well  and overcomes  wavelet reconstruction artifacts present in the standard WVD.

{\rot Note that the operators defined by \eqref{eq:svd-filt} are linear, and as a result, the TI-DFD belongs the the class of  linear regularization method after a suitable parameter choice is made \cite{engl1996regularization}. However, it is distinct from existing instances and our results do not follow from previously known results in that area. Furthermore, optimal convergence rates are dependent on both the chosen method and the set of solutions under consideration. Among other, interesting extensions of our theory involve the filtered TI-DFD with non-linear filters \cite{donoho1995nonlinear} or the analysis in the presence of stochastic noise  \cite{zhang2022stochastic,pereverzev2013recommendedlegendre}.}

\subsection{Outline}

The remainder of this paper is organized as follows. In Section \ref{sec:dfd}  we  introduce the novel concept of  TI-DFD and provide some of its properties. In Section \ref{sec:reg} we introduce and analyze the proposed  regularization concept in the form of the filtered TI-DFD. We prove  regularization properties and  derive order optimal convergence rates  for an a-priori parameter choice rule. In Section \ref{sec:example}  we construct a TI-DFD for the 1D integral operator for which we support  our theory  by numerical simulation. The paper concludes with a short summary and outlook given in Section~\ref{sec:conclusion}.

\section{TI-DFD of linear operators}
\label{sec:dfd}

Throughout this paper, let $\Ko \colon \dom(\Ko) \subseteq L^2(\R^d) \to \Y$ be a linear, closed and not necessarily bounded operator between  $L^2(\R^d)$ and another Hilbert space $\Y$.

We write $u^\ast \kl{x} \coloneqq \overline{u (-x)}$, where $\overline{z}$ is  the complex conjugate of $z \in \C$. The Fourier transform of  $ \signal \in L^2(\R^d)$ is denoted by $\hat \signal = \Fo \signal $ where $\hat \signal(\xi) \coloneqq \int_{\R^d} f(x) e^{-i \inner{\xi,x}} \diff x$ for integrable functions.   Recall that the Fourier transform turns  convolution into multiplication. In particular for $u, \signal \in L^2(\R^d)$ with $\hat u \in L^\infty(\R^d)$,  the convolution $u^* \ast \signal \in L^2(\R^d)$  is  well defined  by  $(u^* \ast \signal) (x) \int_{\R^d} = \int_{\R^d} u^* (t) \signal(x - t) \diff t $ and satisfies
$u^\ast \ast \signal = \Fo ^{-1} ( (\overline{\Fo u})  \cdot (\Fo \signal) )$. Note that $\tau_a (u \ast \signal) = u \ast (\tau_a \signal)$ where $(\tau_a u )(x) = u (x - a)$ with $a \in \R^d$ is the translation operator and thus the  convolution itself is translation invariant.

\subsection{TI-frames}
\label{sec:TI-frame}

In this subsection we recall the concept of TI-frames and collect properties we will require for our purpose.  Related issues can found in \cite[Sec. 5.2]{mallat}.

\begin{definition}[TI-frame] \label{def:ti-frame}
Let $\Lambda$ be an at most countable index set. A family $ (u_\la)_{\la  \in \La} \in L^2(\R^d)^\La$ is called a translation invariant frame (TI-frame)  for $L^2(\R^d)$ if $\hat u_\la \in L^\infty(\R^d) $ for all $\la \in \La$ and there exist constants $A,B>0$ such that
\begin{linenomath*}
\begin{equation}\label{eq:TI-frame}
\forall f\in L^2(\R^d) \colon \quad A \norm{\signal}^2 \leq \sum_{\la  \in \La} \norm{ u^\ast_\la \ast \signal}^2 \leq B  \norm{\signal}^2  \,.
\end{equation}
\end{linenomath*}
A TI-frame  $ (u_\la)_{\la  \in \La}$ is called tight if  \eqref{eq:TI-frame} holds with TI-frame founds $A=B=1$.
\end{definition}

Because $\hat u_\la \in L^2(\R^d) \cap L^\infty(\R^d)$ and $\hat f \in L^2(\R^d)$, we have $ u^\ast_\la \ast \signal = \Fo^{-1} ( (\overline{ \Fo u_\la}) \cdot (\Fo f)) \in L^2(\R^d) $. From Plancherel's theorem we  get $\norm{ u^\ast_\la \ast \signal}^2  = (2\pi)^{-1} \int_{\R^d} \abs{ \Fo u_\la}^2 \abs{\Fo f}^2$. The right inequality  in the TI-frame property \eqref{eq:TI-frame} therefore in particular implies  $\kl{u^\ast_\la \ast \signal}_{\la  \in \La} \in \ell^2 \kl{\La, L^2(\R^d)}$. Here and  below  for a  Hilbert space $\Hr$ we write $\ell^2 \kl{\La, \Hr}$ for the Hilbert space of all  $\mathbf{c} = (c_\la)_{\la \in \La} \in  \Hr^\La$ satisfying $\norm{\mathbf{c}}_\La^2 \coloneqq \sum_{\la  \in \La} \norm{c_\la}^2 <  \infty$ with  inner product $\inner{\mathbf{c}, \mathbf{d}}_\Lambda \coloneqq \sum_{\la  \in \La} \inner{c_\lambda, d_\lambda}$.

Along with the definition of a TI-frame, it is convenient to introduce the following TI  versions of synthesis, analysis and frame operators.

\begin{definition}[TI-synthesis, TI-analysis and TI-frame operator]\label{def:frame-operators}
Let $(u_\la)_{\la \in \La}$ be a TI-frame.
\begin{enumerate}[label=(\alph*),topsep=0em]
\item
$\Uo \colon \ell^2 \kl{\La, L^2(\R^d)} \to L^2(\R^d) \colon  \kl{c_\la}_{\la  \in \La}\mapsto \sum_{\la  \in \La} u_\la \ast c_\la $ is called TI-synthesis operator.
\item
$\Uo^* \colon L^2(\R^d) \to \ell^2 \kl{\La, L^2(\R^d)} \colon  \signal \mapsto \kl{u^\ast_\la \ast \signal}_{\la  \in \La} $ is called TI-analysis operator.
\item
$\Uo\Uo^* \colon  L^2(\R^d) \to  L^2(\R^d)$ is called  TI-frame operator.
\end{enumerate}
\end{definition}

Note that the TI-analysis operator is the adjoint of the TI-synthesis operator. Using the definition of the TI-analysis operator and the norm on  $\ell^2 \kl{\La, L^2(\R^d)}$, we  can rewrite the frame condition \eqref{eq:TI-frame} as $ \forall f\in L^2(\R^d) \colon A \norm{\signal} \leq \norm{\Uo^* \signal}_\Lambda^2\leq B \norm{\signal} $. The right inequality  in \eqref{eq:TI-frame} states that the TI-analysis operator $\Uo^*$ is a well-defined  bounded linear operator, and the left inequality states  that $\Uo^*$ is  bounded  from below.  In particular, its  Moore-Penrose inverse $(\Uo^*)^\ddag \colon \ell^2 \kl{\Lambda, L^2(\R^d)} \to  L^2(\R^d)$ is  bounded and given by $ (\Uo^*)^\ddag = (\Uo\Uo^*)^{-1} \Uo$. Finally,  $A =  \norm{(\Uo^*)^\ddag}^{-2}$ and $B =  \norm{\Uo^*}^{2}$ are the optimal TI-frame bounds for \eqref{eq:TI-frame}.

Let us collect some further useful properties.

\begin{proposition}[Properties  of TI-frames] \label{prop:ti}Let $(u_\la)_{\la  \in \La} \in L^2(\R^d)^\La$.
\begin{enumerate}[label=(\alph*),topsep=0em]
\item \label{prop-ti1} $(u_\la)_{\la  \in \La}$ TI-frame with bounds $A,B>0$ $\iff$ $2\pi A \leq \sum_{\lambda\in\Lambda} \abs{\hat u_\la}^2 \leq 2\pi B$.
\item\label{prop-ti2} $(u_\la)_{\la  \in \La}$ tight TI-frame  $\iff$ $\sum_{\lambda\in\Lambda} \abs{\hat u_\la}^2 =2\pi$.
\item\label{prop-ti3}
$(u_\la)_{\la  \in \La}$ TI-frame $\Rightarrow$  $(u_\mu^\ddag)_{\mu \in \La}$ with $ u_\mu^\ddag \coloneqq  \Fo^{-1}( 2\pi \hat u_\mu / \sum_{\la \in \Lambda} \abs{\hat u_\la}^2 )$ is a TI-frame.

\item\label{prop-ti4}  $(\Uo^*)^\ddag$  is the synthesis operator of  $(u_\la^\ddag)_{\la \in \La}$.

\item\label{prop-ti5}  $(u_\la)_{\la  \in \La}$  TI-frame $\Rightarrow$  $\id = (\Uo^*)^\ddag \Uo^*$.

\item\label{prop-ti6}  $(u_\la)_{\la  \in \La}$ tight TI-frame $\Rightarrow$   $ \id = \Uo \Uo^*$.

\end{enumerate}
\end{proposition}

\begin{proof}
Items  \ref{prop-ti1},  \ref{prop-ti2} follow from \eqref{eq:TI-frame} and the Plancherel theorem. Item \ref{prop-ti3}, \ref{prop-ti4} follow from \ref{prop-ti1} and  straight forward computations. Items \ref{prop-ti5}, \ref{prop-ti6} hold because the Moore-Penrose inverse  is a right inverse  and $(\Uo^*)^\ddag = \Uo$  for TI-tight frames.\end{proof}

\begin{definition}[Dual TI-frame]
Let $(u_\la)_{\la \in \La}$, $(w_\la)_{\la \in \La}$ be TI-frames  with TI-sythesis operators $\Uo$ and $\Wo$. We call $(w_\la)_{\la \in \La}$ a dual TI-frame to  $(u_\la)_{\la \in \La}$ if $\Wo \Uo^* =  \id$.
\end{definition}

On easily verifies that  dual frame condition is equivalent to  the identity $ \sum_{\lambda} (\Fo w_\la) \cdot ( \overline{\Fo u_\la}) = 2\pi$. Moreover, in this case,
\begin{linenomath*}
\begin{equation} \label{eq:inv-dual}
	\forall \signal \in L^2(\R^d) \colon \quad
	\signal = \Wo \Uo^* \signal = \sum_{\la  \in \La} w_\la \ast (u^\ast_\la \ast \signal)  \,.
\end{equation}
\end{linenomath*}
The reproducing formula  \eqref{eq:inv-dual}  in particular holds for  $(u_\la^\ddag)_{\la \in \La} = (w_\la)_{\la \in \La} $, which is referred to as the canonical  dual TI-frame. From  \ref{prop-ti1}, \ref{prop-ti3}  it follows that that frame bounds $A$ and $B$ for $(u_\la)_{\la\in\La}$ give frame bounds $B^{-1}$ and $A^{-1}$ of $(u_\la^\ddag)_{\la \in \La}$. According  to the characterization  via the Moore-Penrose inverse, the canonical dual applied to coefficients $\mathbf{c} = (c_\la)_{\la \in \La} \in \ell^2(\La, L^2(\R^d))$ is characterized as minimizer of the least square functional  $ f \mapsto \norm{\Uo^* \signal - \mathbf{c}}^2$.  In some applications other left inverses  and dual TI-frames may be of interest.

We conclude this subsection by drawing some connections to classical frames. This will be part  of our reasoning for introducing  TI-DFDs as additional tool besides the  DFD and SVD.

\begin{remark}[Frames versus  TI-frames] \mbox{}
We again point out that a TI-frame is not a frame in the classical sense.  A family $ (u_{\la, k})_{(\la, k)  \in \La \times \Z^n } \in  L^2(\R^d)^{\La \times \Z^n}$ for some parameter $n \in \N$ is a frame of $L^2(\R^d)$ (in the classical sense)  if there exist frame bounds $A, B \in (0, \infty)$ such that
\begin{linenomath*}
\begin{equation} \label{eq:frame}
	\forall \signal  \in L^2(\R^d) \colon \quad
	A \norm{\signal}^2 \leq \sum_{\la \in \La}\sum_{k  \in  \Z^n } \abs{\inner{ \signal , u_{\la, k}}}^2 \leq B  \norm{\signal}^2 \,.
\end{equation}
\end{linenomath*}
Note that we use the index set $\La \times \Z^n$  for the frame in order to  make the comparison with TI frames more obvious and to be closer to the notion of wavelet frames. The identity $u^\ast_\la \ast \signal (x)  = \inner{ \signal  ,  u_\la ( (\,\cdot\, ) - x)}  $ shows that a TI-frame may be seen as a generalized notion of a frame using the semi-discrete index set $\La \times \R^d$, frame elements $ u_{\la,x} = u_\la ( (\,\cdot\, ) - x)$  and the squared $L^2$-norm replacing the inner $\ell^2$-norm in \eqref{eq:frame}.

Conversely, classical  frames can be obtained from TI-frames by discretizing the convolutions $u^\ast_\la \ast \signal $. As an example in the context of wavelet analysis, consider a one dimensional mother wavelet $u \in L^2(\R)$ and set $u_j(x)  \coloneqq   2^j u(2^j x) $.  The family $(u_{j})_{j\in \Z}$ is a TI-frame if $A \leq \sum_{j \in \Z} \abs{\hat u(2^{-j} \omega)}^2 \leq B$.  Under additional assumptions \cite{tenlecturesonwavelets} the   family $(u_{j,k})_{j,k \in \Z}$ defined  by $u_{j,k} (x) =  2^{-j/2}  u_j(x- 2^j k)$ becomes  a frame of  $L^2(\R)$. Note that the  sampling step size $2^j$ for constructing a frame depends on the scale $j$, which in particular destroys the translation invariance of the underlying TI-frame.
\end{remark}

\subsection{Review: diagonal frame decomposition}

Before  actually introducing   TI-DFD in the next subsection, we first review the WVD and more  general DFDs of linear operators. The WVD was introduced in \cite{donoho1995nonlinear} for several integral operators.  The  more general  notation  of a DFD used below is taken from \cite{ebner2020regularization}.

 \begin{definition}[DFD]\label{def:fdc} The system $ \kl{u_{\la,k}, v_{\la,k}, \qsv_{\la,k}}_{(\la, k)  \in \La \times \Z^{n} }$, for $n\in \N$ and countable $\La$, is called DFD for $\Ko$, if the following properties hold:
\begin{enumerate}[itemindent=2em, leftmargin=3em,  label=(DFD\arabic*), topsep=0em]
\item \label{dfd1} $(u_{\la,k})_{\la,k \in \La \times \rot \Z^n}$ is a frame of  $L^2(\R^d)$.
\item \label{dfd2} $(v_{\la,k})_{\la,k \in \La \times \rot \Z^n}$ is a frame of  $\overline{\ran\Ko}$.
\item \label{dfd3}  $ \forall (\la, k) \in \La \times \Z^n \colon  \Ko^* v_{\la, k} = \qsv_\la  u_{\la, k}$.
\end{enumerate}
 \end{definition}

If $\kl{u_{j, k}}_{(j,k) \in \Z \times \rot \Z^n}$ is a wavelet frame  we refer to  the DFD as  WVD. The elements $v_{\la, k}$ in this case  are called vaguelettes. With  a  dual frame $(w_{\la, k})_{(\la,k) \in \La \times \Z^n }$, the following reproducing formula holds for all $f \in \dom (\Ko)$:
\begin{linenomath*}
\begin{equation}\label{eq:wvd-rep}
\signal = \sum_{\la   \in \La} \frac{1}{\qsv_\lambda} \sum_{k \in \Z^{n}}  \inner{\Ko \signal, v_{\la,k}}  w_{\la,k} \,.
\end{equation}
\end{linenomath*}
The DFD includes the SVD, in which case $n=0$  and  $(u_{\la,0})_{\la \in \La}$ and  $(v_{\la,0})_{\la \in \La}$ are orthonormal bases, and the WVD  where $n=d$ and  $\kl{u_{\la, k}}_{(\la,k) \in \La \times \Z^d}$ is a wavelet frame. As mentioned in the introduction the WVD  can be better adapted to the signal class compared to the SVD.

Wavelet frames suffer from specific artifacts  after coefficient filtering, mainly due to the lack of translation invariance.  Therefore we will develop a related concept  using TI-frames instead of frames.   In order to further motivate our approach, below  we give a representation of the WVD   using  sampled convolutions and point out where TI invariance can be restored.

\begin{remark}[WVD in sampled convolution form] \label{rem:wvd-conv}
Consider a 1D orthonormal wavelet basis $(u_{j,k})_{j,k \in \Z} $ of $L^2(\R)$ with $ u_{j,k} (x) \coloneqq  2^{-j/2} u ( x - 2^j k ) $ and $u_j(x) \coloneqq 2^j u (2^j x )$ with mother wavelet $u \in L^2(\R)$.   Moreover, let $(u_{j,k},v_{j,k}, \qsv_j)_{j,k  \in \Z}$ be the corresponding WVD for  $\Ko$.  For  $\signal \in L^2(\R)$ and $\mathbf{a} = (a_k)_{k\in \Z} \in \ell^2(\Z)$ define the downsampled and upsampled  convolutions  $	(u^\ast_j \circledast_j f)_k \coloneqq   2^{-j/2}  \cdot  (u^\ast_j \ast  f)     (2^jk )   $ and $(u_j \circledast_j \mathbf{a})(x) \coloneqq   2^{-j/2} \sum\nolimits_{k\in \Z}  a_k u_j (x-k2^j)$ respectively. The WVD reconstruction formula can be written as
\begin{linenomath*}
\begin{align} \label{eq:wvd-conv1}
 \signal   &= \sum_{j\in \Z} u_j \circledast_j  \bigl( u^\ast_j \circledast_j \signal \bigr)  \,,
\\ \label{eq:wvd-conv2}
u^\ast_j \circledast_j \signal &= \frac{1}{\qsv_j} (\inner{\Ko f, v_{j,k}})_{k\in \Z} \,.
\end{align}
\end{linenomath*}
In particular, level-depending downsampling and upsampling destroys the translational invariance of the system $(u_j)_{j \in \Z}$. Note that  \eqref{eq:wvd-conv1}, \eqref{eq:wvd-conv2} is  equivalent to  \eqref{eq:wvd-rep}. The proposal of this paper can be seen as a way to restore translational invariance by replacing the  sampled convolutions  in \eqref{eq:wvd-conv1} by the non-sampled counterparts and to modify \eqref{eq:wvd-conv2} accordingly.
\end{remark}

\subsection{Introducing  the TI-DFD}
\label{ssec:ti-dfd}

We now introduce the TI-DFD as the central concept of this  paper.  We denote by $B(\Y, L^2(\R^d))$ the space of all bounded linear operators  from $\Y$ to $L^2(\R^d)$.

\begin{definition}[TI-DFD]\label{def:ti-dfd} We call the system $ \kl{u_\la, \Vo^*_\la, \qsv_\la}_{\la  \in \La}$ a translation invariant frame decomposition (TI-DFD) for $\Ko$, if the following properties hold:
\begin{enumerate}[itemindent=1em, leftmargin=2em,  label=(TI\arabic*), topsep=0em]

\item \label{ti1} $ \kl{u_\la}_{\la  \in \La} \in L^2(\R^d)^\Lambda$ is a TI-frame for $L^2(\R^d)$.

\item \label{ti2} $\forall \la \in \La$ we have $ \Vo^*_\la \in  B(\Y, L^2(\R^d))$ and $\forall g \in \overline{\ran{\Ko}}\colon \sum_{\la  \in \La} \norm{\Vo^*_\la \data}^2 \asymp \norm{g}^2$.

\item \label{ti3} $\forall \la \in \La \colon  \qsv_\la \in (0, \infty)$ and $\forall f\in \dom(\Ko) \colon  \Vo^*_\la (\Ko \signal)   = \qsv_\la \, (u^\ast_\la \ast \signal ) $.
\end{enumerate}
\end{definition}

Let us compare  a TI-DFD  to a regular DFD $ \kl{u_{\la, k}, v_{\la, k}, \qsv_\la}_{(\la, k) \in \La \times \Z^d}$, where $u_{\la, k}(x) = u_{\la} (x - M_\la  k)$ with sampling matrix $M_\la \in \R^{d \times d}$.   Among others, such forms includes WVDs and DFDs with curvelet or shearlet frames. For example, for  the 1D wavelet transform (Remark~\ref{rem:wvd-conv}) we have $\La =\Z$, and for the 2D wavelet transform we have $\La = \Z \times \{\mathrm{H},\mathrm{V},\mathrm{D}\}$ representing  horizontal (H), vertical (V)  and diagonal (D) wavelets at scale $j\in \Z$.  In such situations,  \ref{ti1}, \ref{ti2} replace the frame conditions \ref{dfd1}, \ref{dfd2}. In item \ref{ti3}, $u^\ast_\la \ast \signal$  takes over the role  of the inner products $\inner{ u_{\la, k} , \signal}  = (u^\ast_\la \ast \signal) (M_\la  k)$ and  $\Vo^*_\la \data $ takes over the role of  $(\inner{ v_{\la, k} , g})_{k \in \Z^d}$.  Finally,  the identity  $ \Vo^*_\la \Ko \signal = \qsv_\la \, u^\ast_\la \ast \signal $ replaces the quasi-singular value relation $\inner{ v_{\la, k} , \Ko f} = \qsv_\la \inner{ u_{\la, k} , \signal} $. In particular, the standard DFD may be seen as kind of discretization of the TI-frame decomposition, where the discretization depends on the index $\la$.

\begin{remark}[Operator $\Vo^*$]
According to  \ref{ti2} there exist constants $A, B \in (0, \infty)$ such that   $ A  \norm{\data}^2  \leq \sum_{\la  \in \La} \norm{\Vo^*_\la \data}^2  	\leq 	B \norm{\data}^2$ for all $g \in \overline{\ran{\Ko}}$. Equivalently, the operator
\begin{linenomath*}
\begin{equation} \label{eq:V}
\Vo^*   \colon \Y \to \ell^2(\La, L^2(\R^d)) \colon f \mapsto (\Vo^*_\la f)_{\la \in \La}
\end{equation}
\end{linenomath*}
is well defined, bounded and  bounded from below  and above with norm bounds $\norm{\Vo^*} \leq B^{1/2}$ and   $\norm{(\Vo^*)^\ddag} \leq A^{-1/2}$. It plays  plays a similar role in $\Y$ as the TI-analysis operator $\Uo^*$ does in  $L^2(\R^d)$. However, operator $\Vo^*$ is  not of the convolution form in general.
\end{remark}

Similar to the standard  DFD we have the  following TI-DFD reconstruction formula.

\begin{proposition}[Exact reconstruction formula  via TI-DFD]
Let $(u_\la,\Vo^*_\la, \qsv_\la)_{\la \in \La}$ be a TI-DFD for $\Ko$ and $(w_\la)_{\la \in \La}$ be a dual TI-frame for  $(u_\la)_{\la \in \La}$. Then
\begin{linenomath*}
\begin{equation}\label{eq:ti-recon}
\forall \data \in \ran(\Ko) \colon  \quad \Ko^{-1}\data = \sum_{\lambda\in\Lambda}   w_\lambda  \ast ( \qsv_\lambda^{-1} \cdot (\Vo^*_\la \data) ) = \Wo \bigl( (\qsv_\la^{-1}  \Vo^*_\la \data)_{\la \in \La}  \bigr) \,.
\end{equation}
\end{linenomath*}
\end{proposition}

\begin{proof}
With  $ \data  = \Ko \signal$, identity \eqref{eq:ti-recon} follows  after inserting  \ref{ti3} into~\eqref{eq:inv-dual}.
\end{proof}

Similar  to the SVD and the  DFD reconstruction formulas,  \eqref{eq:ti-recon} reflects the ill-posedness of $\Ko^{-1}$ in terms of the quasi-singular values that are potentially accumulating  at zero.  More precisely, we have the following results.

\begin{theorem}[Characterization of ill-posedness via TI-DFD]\label{thm:char}
Let $(u_\la,\Vo^*_\la, \qsv_\la)_{\la \in \La}$ be a TI-DFD of $\Ko$. Then  the following hold.
\begin{enumerate}[label=(\alph*), topsep=0em]
\item\label{thm:char1} $\inf_{\la \in \La}  \qsv_\la >0$ $\Rightarrow$  $\Ko^{-1}$ is bounded.
\item\label{thm:char2} $ \inf_{\la \in \La}  \qsv_\la =0$ and   $\inf_\la \norm{\Vo^*_\la} >0$
$ \Rightarrow $ $\Ko^{-1}$  unbounded.
\end{enumerate}
\end{theorem}

\begin{proof}
{According to the reconstruction formula \eqref{eq:ti-recon}  we have the identities $\Ko^{-1}\data = \Wo \bigl( (\qsv_\la^{-1}  \Vo^*_\la \data)_{\la \in \La}  \bigr)$ and $ \Wo^\ddag  \Ko^{-1}\data = (\qsv_\la^{-1}  \Vo^*_\la \data)_{\la \in \La}$. Therefore, by definition of the operator norm,
\begin{linenomath*}
\begin{equation} \label{eq:char-aux}
\frac{ \norm{(\qsv_\la^{-1}  \Vo^*_\la)_{\la \in \La}(g)}}{\norm{\Wo^\ddag}}
\leq \norm{\Ko^{-1} g}   \leq \frac{\norm{\Wo} \norm{\Vo^*} \norm{g} }{(\inf_{\la \in \La}  \qsv_\la)}   \,.
\end{equation}
\end{linenomath*} }
The  right inequality  gives  \ref{thm:char1}. To  verify Item~\ref{thm:char2}, suppose  $\inf_{\la \in \La}  \qsv_\la = 0$ and $\inf_\la \norm{\Vo^*_\la} >0$.   Because   $\norm{(\qsv_\la^{-1}  \Vo^*_\la)_{\la \in \La}} \geq   \abs{\qsv_\mu}^{-1} \norm{ \Vo^*_\mu}$  for all  $\mu \in \La$ this implies that $(\qsv_\la^{-1}  \Vo^*_\la)_{\la \in \La}$ is unbounded.  Together with the left inequality in \ref{eq:char-aux} this shows  the  unboundedness of $\Ko^{-1}$.
\end{proof}

In summary, under the reasonable assumption that $\inf_\la \norm{\Vo^*_\la} >0$, the inverse operator  $\Ko^{-1}$ is unbounded  if and only  if the quasi-singular values $\qsv_\la$ accumulate at zero.

We note that for a stable inverse problem a TI-DFD has been constructed in \cite[Theorem 3.5]{zangerl2021multiscale} in the form of a convolution factorization for the wave equation. An example for a  TI-DFD in the ill-posed situation  will be  constructed in Section~\ref{sec:example} for the 1D integration operator.

\section{Regularization by filtered TI-DFD}
\label{sec:reg}

Throughout this section, let $(u_\la,\Vo^*_\la, \qsv_\la)_{\la \in \La}$ be a TI-DFD for  $\Ko$, and $(w_\la)_{\la \in \La}$ be a dual frame of $(u_\la)_{\la \in \La}$.  Typically,   solving inverse problems of the form  \eqref{eq:ip} is unstable (see Theorem \ref{thm:char}) and hence need to be  regularized.   For that purpose, we introduce and analyze the concept of filtered TI-DFD in this section.

\subsection{Definition of the filtered TI-DFD}

We first  recall  the definition  of  regularizing filters. We adopt  the Definition from  \cite{ebner2020regularization}  which is slightly more general than the standard definition   \cite{engl1996regularization}.

\begin{definition}[Regularizing filter] \label{def:regfilt} A family $(\filt_\al)_{\alpha>0}$ of piecewise continuous functions $\filt_\al \colon (0, \infty) \to \R$ is called a regularizing filter if the following hold:
\begin{enumerate}[itemindent=2em, leftmargin=2em,  label=(F\arabic*), topsep=0em]
\item\label{def:regfilt1}  $\forall \alpha >  0 \colon \norm{\filt_\al}_\infty < \infty$.
\item\label{def:regfilt2} $\exists C > 0 \colon \sup \{\abs{\qsv \filt_\al(\qsv) } \colon \alpha > 0 \wedge \qsv  \geq 0\} \leq C$.
\item\label{def:regfilt3} $\forall \qsv \in (0, \infty) \colon \lim_{\alpha \to 0} \filt_\al (\qsv) = 1/\qsv$.
\end{enumerate}
\end{definition}

Using the concept of regularizing filters, we study the following filtered versions of the TI-DFD reconstruction formula.

\begin{definition}[Filtered TI-DFD] \label{def:filteredTIDFD}
Let $(\filt_\al)_{\alpha>0}$ be a regularizing filter. We call the family  $(\reg^\filt_\al)_{\alpha>0}$  of operators $\reg^\filt_\al \colon \Y  \to L^2(\R^d) $ defined by
\begin{linenomath*}
\begin{equation} \label{eq:ti-filt}
\reg^\filt_\al  g \coloneqq  \sum_{\lambda\in\Lambda}  w_\lambda \ast ( \filt_\al (\qsv_\lambda)  \cdot  (\Vo^*_\la g ) )   =    \Wo  \bigl( (\filt_\al (\qsv_\lambda) \cdot   \Vo^*_\la \data)_{\la \in \La}  \bigr)
\end{equation}
\end{linenomath*}
the filtered TI-DFD according to the filter $(\filt_\al)_{\alpha>0}$ and the TI-DFD $(u_\la,\Vo^*_\la, \qsv_\la)_{\la \in \La}$.
\end{definition}

We first show the  well-posedness of filtered TI-DFD.

\begin{proposition}[Existence and stability] \label{prop:well}
Let $(\filt_\al)_{\alpha>0}$ be a regularizing filter.
For any $\alpha >0$, the operator $\reg^\filt_\al$ as in \eqref{eq:ti-filt} is well defined, linear  and bounded with  $\norm{\reg^\filt_\al} \leq \norm{\filt_\al}_\infty \norm{\Wo}\norm{\Vo^*} $.
\end{proposition}

\begin{proof}
Fix $\alpha >0$ and let $\data \in \ran (\Ko)$. Clearly $\reg^\filt_\al$ is a linear operator. From the upper frame property of $(w_\la)_{\la \in \La}$, the boundedness of the filters $\filt_\al$ (see \ref{def:regfilt1})  and the boundedness of $\Vo^*$ we obtain $\norm{\reg^\filt_\al  g}  =  \lVert \Wo \bigl( (\filt_\al (\qsv_\lambda) \cdot   \Vo^*_\la \data)_{\la \in \La}  \bigr) \rVert
\leq \norm{\filt_\al}_\infty \norm{\Wo}   \norm{\Vo^*} \norm{g}$. This shows, that $\reg^\filt_\al$ is well defined and  gives the claimed norm estimate.
\end{proof}

\subsection{Convergence analysis}

For  the following, let $(\filt_\al)_{\alpha>0}$ be a regularizing filter and $\kl{\reg^\filt_\al}_{\alpha > 0}$ be the filtered TI-DFD defined in \eqref{eq:ti-filt}. In what follows, we show that the filtered TI-DFD yields a regularization method. To this end,  we first recall the definition of a regularization method.

\begin{definition}
Let   $(\reg_\alpha)_{\alpha>0}$ be a family of bounded linear operators $\reg_\alpha \colon \Y \to L^2(\R^d)$, let $g \in \ran (\Ko)$ and $\tilde\al \colon (0, \infty) \times \Y \to (0, \infty)$. The pair $\kl{\kl{\reg_\alpha}_{\alpha > 0}, \tilde\al}$ is a regularization method for the solution of $\Ko \signal = g$, if
\begin{linenomath*}
\begin{align*}
& \lim_{\delta \to 0} \sup \bigl\{\tilde\al\kl{\delta, \data^\delta} \mid \data^\delta \in \Y \wedge \norm{\data^\delta - g} \leq \delta \bigr\} = 0 \,,\\
& \lim_{\delta \to 0} \sup \bigl\{\norm{f - \reg_{\tilde\al\kl{\delta, \data^\delta}}\data^\delta} \mid \data^\delta \in \Y \wedge \norm{\data^\delta - g} \leq \delta \bigr\} =0 \,.
\end{align*}
\end{linenomath*}
In this case, function $\tilde\al$ is called admissible parameter choice.
\end{definition}

As an auxiliary result we show  convergence as $\al \to 0$ for exact data.

\begin{proposition}[Pointwise convergence] \label{prop:pointwise} For all $\data \in \ran (\Ko)$ we have $\lim_{\alpha \to 0} \reg^\filt_\alpha \data = \Ko^{-1}  \data$.
\end{proposition}

\begin{proof}
Let $\data \in \ran (\Ko)$ and $\signal \in  \dom (\Ko)$ with  $\Ko \signal = \data$. From the  reproducing formula \eqref{eq:inv-dual} we have  $\signal = \sum_{\la  \in \La} w_\lambda  \ast u^\ast_\la \ast \signal$ and from the definition of the filtered TI-DFD together with \ref{ti3} we have $\reg^\filt_\al \data =  \sum_{\lambda\in\Lambda} w_\lambda  \ast ( \filt_\al (\qsv_\lambda)  \qsv_\lambda \cdot  u^\ast_\la \ast \signal)$. As a consequence,
\begin{linenomath*}
\begin{equation}\label{eq:pointwiseaux}
\norm{\signal - \reg^\filt_\al \data}^2
\leq
\norm{\Wo}^2 \sum_{\la \in \La} \abs{1 -  \filt_\al \kl{\qsv_\la} \qsv_\la}^2 \norm{u^\ast_\la \ast f}^2
\leq
\norm{\Wo}^2\norm{\Uo^*}^2  \sup_{\la  \in \La}  \abs{1 -  \filt_\al \kl{\qsv_\la} \qsv_\la}^2 \norm{\signal}^2 \,.
\end{equation}
\end{linenomath*}
According to  \ref{def:regfilt3}, $\lim_{\al \to 0} \abs{1 -  \filt_\al \kl{\qsv_\la} \qsv_\la  } =   0$ pointwise and according to  \ref{def:regfilt2},  $\sup_{\la  \in \La}  \abs{1 -  \filt_\al \kl{\qsv_\la} \qsv_\la }$ is bounded independently  of $\alpha$. Application of the dominated convergence theorem to the sum in \eqref{eq:pointwiseaux} yield $\norm{f - \reg^\filt_\al g} \to 0$.
\end{proof}

As a consequence  of Propositions~\ref{prop:well} and \ref{prop:pointwise} we derive the following main regularization  result. For convenience of the reader we restate all assumptions on the operator and the filtered DFD in that theorem.

\begin{theorem}[Filtered TI-DFD is regularization method] \label{thm:convergence}
Let $(u_\la,\Vo^*_\la, \qsv_\la)_{\la \in \La}$ be a TI-DFD for  the closed linear operator $\Ko \colon \dom(\Ko) \subseteq L^2(\R^d)  \to \Y$,  $(w_\la)_{\la\in \La}$ a dual frame of $(u_\la)_{\la\in \La}$, $(\filt_\al)_{\alpha>0}$ a regularizing filter and $(\reg^\filt_\al)_{\alpha>0}$ defined by \eqref{eq:ti-filt}. Let  $\tilde\al \colon (0, \infty) \to (0, \infty)$ satisfy
\begin{linenomath*}
\begin{equation} \label{eq:convergence-rule}
\lim_{\delta \to  0} \tilde\al (\delta) =  \lim_{\delta \to  0}  \delta \norm{\filt_{\tilde\al (\delta)}}_\infty =0\,.
\end{equation}
\end{linenomath*}
Then  $\kl{(\reg^\filt_\al)_{\alpha>0}, \tilde\al}$ is a regularization method for $\Ko \signal = \data$ for any $\data \in \ran (\Ko)$.
\end{theorem}

\begin{proof}
According to Propositions \ref{prop:well} and \ref{prop:pointwise}, $\kl{\filt_\al}_{\alpha > 0}$ is a family of bounded linear operators that converges point-wise to $\Ko^{-1}$ on $\dom \kl{\Ko^{-1}}$. Using \cite[Propositions 3.4, 3.7]{engl1996regularization} the pair $\kl{\kl{\reg^\filt_\al}_{\alpha > 0}, \tilde\al}$ is a regularization method if $\tilde\al (\delta), \delta \norm{\reg_{\tilde\al (\delta)}} \to 0$ as $\delta \to 0$. The estimate $\norm{\reg^\filt_\al} \leq  \norm{\filt_\al}_\infty \norm{\Wo} \norm{\Vo^*}  $ derived in Proposition \ref{prop:well} then yields the claim.
\end{proof}

\subsection{Convergence rates}

As the next result we derive convergence rates for filtered TI-DFD, which gives quantitative estimates for  the  reconstruction error $\norm{\signal_\star - \reg_\alpha \data^\delta}$. Due to the ill-posedness of \eqref{eq:ip} such estimates require additional assumptions on the exact unknown for any reconstruction method. In the following we write $\mathbf{h} = (h_\la)_{\la\in \La} \in \ell^2\kl{\Lambda, L^2(\R^d)}$. We  derive convergence  rates under the following assumptions on the exact unknown $\signal_\star \in L^2(\R^d)$ and the regularizing filter $(\filt_\al)_{\alpha>0}$.

\begin{assumption}[Convergence rates conditions]
For $\mu, \rho >0$ suppose:
\begin{enumerate}[itemindent=2em, leftmargin=2em,  label=(R\arabic*), topsep=0em]
\item \label{rate-sc}$\exists \mathbf{h}  \in \ell^2\kl{\Lambda, L^2(\R^d)}\colon  \norm{\mathbf{h}} \leq \rho$  $ \wedge$ $\forall \lambda \in \Lambda \colon u^\ast_\la \ast \signal_\star = \qsv_\lambda^{2\mu} h_\lambda$.
\item \label{rate-f1} $\norm{\filt_\al}_{\infty} = \mathcal{O}\kl{\alpha^{-1/2}}$ as $\alpha \to 0$.
\item \label{rate-f2} $\forall \alpha > 0 \colon \sup\{\qsv^{2\mu} \abs{1-\qsv \filt_\al \kl{\qsv}} \mid \qsv \in (0, \infty)\} \leq C_\mu \alpha^\mu $.
\end{enumerate}
\end{assumption}

Source condition \ref{rate-sc} is an abstract smoothness condition relating the element to be reconstructed with the ill-posedness of the operator characterized by the quasi-singular values. Conditions  \ref{rate-f1}-\ref{rate-f2} restrict the class of filters yielding  a desired rate. For example,  it is well known that  the hard truncation filter satisfies these conditions for all $\mu>0$.

We have the following result.

\begin{theorem}[Convergence rates]\label{thm:rates}
Suppose  $\signal_\star$ and $(\filt_\al)_{\alpha>0}$ satisfy \ref{rate-sc}-\ref{rate-f2}  and make the parameter choice  $\alpha = \tilde\al \kl{\delta, \data^\delta} \asymp \kl{\delta / \rho}^{2/(2\mu+1)}$. Then there exists a constant $c_\mu$  independent of $\signal_\star, \rho$  such that for all $\data^\delta \in \Y$ with $\norm{\data^\delta - \Ko \signal_\star} \leq \delta$,
\begin{linenomath*}
\begin{equation} \label{eq:rates}
	 \norm{\signal_\star - \reg_{\tilde\al}^\filt (\data^\delta)}
	\leq c_\mu \delta^{\frac{2\mu}{2\mu+1}}\rho^{\frac{1}{2\mu+1}} \,.
\end{equation}
\end{linenomath*}
\end{theorem}

\begin{proof}
From \ref{rate-sc}, \ref{rate-f1} and  the estimate  $\norm{\data^\delta - \Ko \signal_\star} \leq \delta$ we obtain
\begin{linenomath*}
\begin{align*}
\norm{\reg^\filt_\al \data^\delta - \signal_\star}
& \leq \norm{\reg^\filt_\al\kl{\data^\delta - \Ko \signal_\star}} + \norm{\reg^\filt_\al\Ko \signal_\star - \signal_\star}
\\
& \leq \norm{\reg^\filt_\al} \delta + \norm{\sum_{\lambda\in \Lambda} w_\lambda \ast ( (1-\filt_\al\kl{\qsv_\la}\qsv_\la) \cdot u^\ast_\la \ast \signal_\star )}
\\
& \leq  \norm{\filt_\al}_\infty \norm{\Wo} \norm{\Vo} \delta + \norm{\Wo} \bigl(\sum_{\la  \in \La} \abs{1-\filt_\al\kl{\qsv_\la}\qsv_\la}^2 \norm{u_\la^* \ast \signal_\star}^2\bigr)^{1/2}
\\
& \leq \norm{\filt_\al}_\infty \norm{\Wo} \norm{\Vo} \delta + \norm{\Wo} \, \bigl(\sum_{\la  \in \La} \abs{(1-\filt_\al\kl{\qsv_\la}\qsv_\la)\qsv_\lambda^{2\mu}}^2 \, \norm{h_\lambda}^2 \bigr)^{1/2}
\\
& \leq c_1 \alpha^{-1/2} \delta + C_\mu \norm{\Wo}  \alpha^{\mu} \rho \,.
\end{align*}
\end{linenomath*}
With the  parameter choice $\alpha = \tilde\al \asymp \kl{\delta/\rho}^{2/(2\mu+1)}$ this yields  \eqref{eq:rates}.
\end{proof}

\subsection{Order optimality}

We next proof that the convergence rates obtained in Theorem~\ref{thm:rates} are optimal. A discussion of order optimality of regularization methods can for example  be found in \cite[Section 3.2]{engl1996regularization}. While the  methods we use are follow  the  spirit of \cite{engl1996regularization}  our results are different as convergence rates  and order optimality depends on the source specific set where these properties are studied. For $\mathcal{M} \subseteq \dom (\Ko)$ and $\reg \colon \Y \to L^2(\R^d)$ define
\begin{linenomath*}
\begin{align} \label{eq:moduls}
E(\mathcal{M}, \delta, \reg) &\coloneqq \sup \{\norm{\reg (\data^\delta) - \signal} \mid f\in \mathcal{M} \wedge \data^\delta \in \Y \wedge \norm{\Ko \signal - \data^\delta} \leq \delta\} \\ \label{eq:worst}
\epsilon (\mathcal{M}, \delta) &\coloneqq \sup \{\norm{\signal} \mid \signal \in \mathcal{M} \wedge \norm{\Ko f} \leq \delta\} 	\,.
\end{align}
\end{linenomath*}
The quantity  $E(\mathcal{M}, \delta, \reg)$  is the worst case reconstruction error  using the reconstruction method  $\reg$ under the  a-priori assumption  $\signal  \in \mathcal{M}$ and the error bound $\norm{\reg \data^\delta - \signal} \leq \delta$. If $\reg(0)=0$ one readily verifies that $ E(\mathcal{M}, \delta, \reg) \geq \epsilon(\mathcal{M}, \delta)$. A family  $\kl{\reg^\delta}_{\delta > 0}$ is called order optimal on $\mathcal{M}$, if   $E(\mathcal{M}, \delta, \reg^\delta) \leq c\cdot  \epsilon \kl{\mathcal{M}, \delta}$ for some $c>0$ and sufficiently small $\delta$.

Theorem \ref{thm:rates} gives an upper bound for the worst case error of filtered TI-DFD on
\begin{linenomath*}
\begin{equation}\label{eq:sset}
	\mathcal{M}_{\mu,\rho}
	\coloneqq \{ \signal \in \dom(\Ko) \mid  \exists \mathbf{h}  \colon \norm{\mathbf{h}} \leq \rho \wedge \forall \la \colon u^\ast_\la \ast \signal = \qsv_\la^{2\mu} h_\la \} \,.
\end{equation}
\end{linenomath*}
In order to show that filtered TI-DFD is order optimal on $\mathcal{M}_{\mu,\rho}$, we will bound $\epsilon (\mathcal{M}_{\mu,\rho}, \delta)$ from below. Such an estimate is derived in the following Theorem~\ref{eq:optimality}, at least  for a sequence of noise levels tending to zero. Note that $\epsilon (\mathcal{M}_{\mu,\rho}, \delta)$ is  monotonically increasing in $\delta$ and therefore we have  an error bound for any sufficiently small $\delta$.

\begin{theorem}[Order optimality of filtered TI-DFD]\label{eq:optimality}
Let $(u_\la,\Vo^*_\la, \qsv_\la)_{\la \in \La}$ be a TI-DFD for $\Ko$ such that $0$ is an accumulation point  of $(\qsv_\la)_{\la >0}$ and assume there exists    $(e_\la)_{\la\in \La}   \in L^2(\R^d)^\La$   with $u_\la^* \ast   e_{\la'} =0 $ for $\la \neq \la'$ and $\norm{u_\la^* \ast   e_\la}  = 1$.  Then, for some sequence $(\delta_n)_{\in \N}$ with $\delta_n \to 0$,
\begin{linenomath*}
\begin{equation} \label{eq:converse}
\epsilon \kl{\mathcal{M}_{\mu,\rho}, \delta_n} \geq
 \norm{\Uo^*}^{-1}  \norm{(\Vo^*)^\ddag}^{-2\mu/(2\mu+1)} \cdot
\delta_n^{\frac{2\mu}{2\mu+1}} \rho^{\frac{1}{2\mu+1}} \,.
\end{equation}
\end{linenomath*}
\end{theorem}

\begin{proof}
After extracting  a subsequence we can assume $\Lambda = \N$ and that $(\qsv_\la)_{\la \in \N}$ converges to zero. For any  $n \in \N$ define $\signal^{(n)} \coloneqq \rho \qsv_n^{2\mu} e_n$.  Then  $u_\la^* \ast \signal^{(n)} =  \qsv_\la^{2\mu} h^{(n)}_\la$ where $h^{(n)}_\la = 0$ for $\la \neq n$ and $h^{(n)}_n = \rho \cdot (u_n^* \ast  e_n)$ with  $\norm{\mathbf{h}^{(n)}} =\norm{h^{(n)}_\la} = \rho$.   In particular,  $\signal^{(n)} \in \mathcal{M}_{\mu,\rho}$. Moreover,
\begin{linenomath*}
\begin{align*}
\norm{\Ko \signal^{(n)}} &\leq  \norm{(\Vo^*)^\ddag} \bigl(\sum_{\la \in \La} \norm{\Vo^*_\la \Ko \signal^{(n)}}^2 \bigr)^{1/2} = \norm{(\Vo^*)^\ddag} \bigl(\sum_{\la \in \La} \qsv_\la^2  \norm{u_\la^* \ast \signal^{(n)}}^2\bigr)^{1/2}
= \norm{(\Vo^*)^\ddag}  \qsv_n^{2\mu+1}  \rho \,,
\\
\norm{\signal^{(n)}}  &\geq  \norm{\Uo^*}^{-1}  \bigl(\sum_{\la \in \La} \norm{u_\la^* \ast \signal^{(n)}}^2 \bigr)^{1/2} = \norm{\Uo^*}^{-1}  \qsv_n^{2\mu}  \rho \,.
\end{align*}
\end{linenomath*}
With $\delta_n \coloneqq \norm{(\Vo^*)^\ddag}  \qsv_n^{2\mu+1}  \rho$, the above estimates imply   $\norm{\Ko \signal^{(n)}} \leq  \delta_n$ and
\begin{linenomath*}
\begin{equation*}
	\norm{\signal^{(n)}} \geq  \norm{\Uo^*}^{-1}     (   \delta_n  \norm{(\Vo^*)^\ddag}^{-1} \rho^{-1})^{2\mu/(2\mu+1)}   \rho
= \norm{\Uo^*}^{-1}\norm{(\Vo^*)^\ddag}^{-2\mu/(2\mu+1)}    \delta_n^{2\mu/(2\mu+1)} \rho^{1/(2\mu+1)}  \,.
\end{equation*}
\end{linenomath*}
Because we have $\signal^{(n)} \in \mathcal{M}_{\mu,\rho}$ and $\epsilon \kl{\mathcal{M}_{\mu,\rho}, \delta_n} \geq \norm{f^{(n)}} $ which yields \eqref{eq:converse}.
\end{proof}

From Theorem  \ref{eq:optimality} and the monotonicity of $\epsilon \kl{\mathcal{M}, \cdot}$  one gets  $ \epsilon \kl{\mathcal{M}, \delta}  \geq c_1 \sum_n \chi_{(\delta_{n+1}, \delta_n]}(\delta) \cdot \delta_n^{2\mu/(2\mu+1)} $ for some for some $c_1>0$ and $\delta_n  \asymp \qsv_n^{2\mu+1}$ after  ordering the quasi-singular values $\qsv_n$ in a descending order. If $ \qsv_n / \qsv_{n+1}$ remains bounded, this shows $ \epsilon \kl{\mathcal{M}, \delta}  \geq c_2  \delta^{2\mu/(2\mu+1)} $ for some $c_2>0$. Together with Theorem \ref{thm:rates} this shows $E(\mathcal{M}, \delta_n, \reg_{\tilde \al}  ) \leq c_3 \cdot  \epsilon \kl{\mathcal{M}, \delta_n}$  for some $c_3>0$ for the filtered TI-DFD $  \reg_{\tilde \al}$ with a-priori parameter choice which in this case yields an order optimal regularization method.

\subsection{Examples for filtered DFDs}
\label{ssec:examples}

We conclude this section  by giving two representative examples for regularizing filters and corresponding filtered TI-DFD, namely  truncated TI-DFD and  Tikhonov-filtered TI-DFD.  In particular we show that the convergence rates conditions are satisfied for all  $\mu >0$ in case of truncated TI-DFD and for $\mu \leq 1$ in case of Tikhonov-filtered TI-DFD.  Concrete examples of TI-DFDs for the 1D integration operator are discussed  in the following section.

\begin{example}[Truncated TI-DFD] \label{ex:trunc}
For  $\al >0$ consider  the truncation filter
 $\filt^{(1)}_\al$ defined by $\filt^{(1)}_\al(\qsv)   \coloneqq \qsv^{-1} \chi_{[\al^{1/2}, \infty)}(\qsv)$. Clearly, conditions  \ref{def:regfilt1}-\ref{def:regfilt3} are satisfied with $C=1$ and $(\filt^{(1)}_\al)_{\al > 0}$ is a regularizing filter.  Furthermore,  $\sup\{\qsv^{2\mu} \abs{1-\qsv \filt^{(1)}_\al(\qsv)} \mid \qsv  >0\} = \sup\{\qsv^{2\mu} \abs{1-\qsv \filt^{(1)}_\al(\qsv)} \mid \qsv^2 < \al \} = \al^\mu$ for $\al, \mu > 0$.  Hence   \ref{rate-f1}, \ref{rate-f2} are satisfied.  The corresponding  truncated TI-DFD becomes
\begin{linenomath*}
\begin{equation} \label{eq:TDFD}
	\reg^{(1)}_\al(y) \coloneqq \sum_{\qsv_\la^2 \geq \al} w_\la \ast  ( \qsv_\la^{-1} \cdot  (\Vo^*_\la g) )  	 \,.
\end{equation}
\end{linenomath*}
The considerations above allow application of  Theorem~\ref{thm:convergence} yielding convergence,  and Theorem~\ref{thm:rates} providing  the convergence rate $ \norm{\signal_\star - \reg^{(1)}_{\tilde \al} \data^\delta}  = \mathcal{O}  (\delta^{2\mu/(2\mu+1)}  \rho^{1/(2\mu+1)})$ under~\ref{rate-sc}.
 \end{example}

Next we consider the Tikhonov filter that already appeared in the introduction in the context  of classical Tikhonov regularization expressed in terms of the SVD.

\begin{example}[Tikhonov-filtered TI-DFD] \label{ex:tik}
For $\al >0$ consider the Tikhonov filter $\filt^{(2)}_\al ( \qsv ) \coloneqq \qsv/(\qsv^2+\al)$. Then   $ \lvert \qsv  \filt^{(2)}_\al(\qsv )  \rvert = \lvert \qsv^2/(\qsv^2 + \al) \rvert   \leq 1$ and  $\lim_{\al \rightarrow 0} \filt^{(2)}_\al (\qsv) = 1/\qsv$. Further, $ \filt^{(2)}_\al$     is bounded, takes its maximum at $\qsv^2 = \al$ and  $\norm{\filt^{(2)}_\al}_\infty =  \al^{-1/2}/2$. Hence  conditions\ref{def:regfilt1}-\ref{def:regfilt3} are satisfied and $(\filt^{(2)}_\al)_{\al >0}$ is a regularizing filter.
Moreover, one shows that   the convergence rates conditions for  Theorem \ref{thm:rates} are satisfied for $\mu \in (0,1]$. However  \ref{rate-f2} is not satisfied for  $\mu>1$, which means that  the Tikhonov filter has qualification  $\mu = 1$ (see the discussion on \cite[p.  76]{engl1996regularization}).   Above considerations  show that  Tikhonov-filtered TI-DFD
\begin{linenomath*}
\begin{equation} \label{eq:zikDFD}
	\reg^{(2)}_\al( y ) \coloneqq \sum_{\la \in \La}
	 w_\la^* \ast  \Bigl(  \frac{\qsv_\la}{\qsv_\la^2+ \al} \cdot  (\Vo^*_\la g) \Bigr)
\end{equation}
\end{linenomath*}
together with a parameter choice satisfying $\delta^2/ {\tilde \al(\delta)} \to 0 $ yields a  regularization method. Moreover, for elements $\signal_\star$ satisfying the source condition \ref{rate-sc}  and parameter choice $\tilde\alpha  \asymp \delta$ the convergence rate $ \norm{\signal_\star - \reg^{(2)}_{\tilde \al} \data^\delta}  = \mathcal{O}  (\delta^{2\mu/(2\mu+1)}  \rho^{1/(2\mu+1)})$ holds.
\end{example}

Further examples of filtered TI-DFDs can be  constructed via known regularizing filters used  in standard SVD-based regularization methods. This includes filters associated to iterative Tikhonov regularization, the Landweber iteration or asymptotic regularization; see  \cite{engl1996regularization}. Note however, that the  combination of these filters with the TI-DFD results in regularization methods that are different from the classical counterparts. For example, $\reg^{(2)}_\al$ is different from Tikhonov regularization except for the very specific case that  $\Uo^*$ and $\Vo^*$ are unitary.

\section{Application: stable differentiation}
\label{sec:example}

In this section we apply the concept of filtered TI-DFD to stable differentiation, the inverse problem associated to 1D integration. In particular, we use TI-wavelets as underlying TI-frame.  While related approaches  are known for denoising \cite{coifman1995translation,nason1995stationary},  we are not aware of  such methods in the context of general inverse problems.  We consider the one-dimensional integration operator as an unbounded operator on  $L^2(\R)$. While the integration operator would be bounded on a bounded domain, working on $\R$ allows to use the concept of TI wavelets and moreover preserves  the translation  invariance of integration.

\subsection{Integration operator on $L^2(\R)$}

Let  $C_\diamond(\R)$ denote the space of all continuous  functions   with compact support and zero integral $\int_\R f =0$. For  functions $\signal \in C_\diamond(\R)$ define the primitive $\Io_\diamond  f  \colon \R \to \C$ by
\begin{linenomath*}
\begin{equation}\label{eq:def K}
	\forall  x \in  \R \colon \quad (\Io_\diamond  f) (x) \coloneqq  \int_{-\infty}^x f(t) \diff t \,.
\end{equation}
\end{linenomath*}
Note that for all $\signal \in C_c(\R)$, the space of all continuous functions with compact support, the primitive $x \mapsto \int_{-\infty}^x f(t) \diff t$ becomes  constant for sufficiently  large $x$. Therefore, the additional assumption of zero integral is necessary and sufficient for the  primitive  being square integrable.

\begin{lemma}[Integration operator on $C_\diamond(\R)$] \label{lem:I}
Operator  $\Io_\diamond \colon C_\diamond(\R) \subseteq L^2(\R) \to L^2(\R) \colon f \mapsto \Io_\diamond f$ is well defined,  linear, densely defined and unbounded.  Moreover, for all $f \in C_\diamond(\R)$, functions $\Fo f$  and $\Fo \Io_\diamond  f$ are continuous with $(\Fo\Io_\diamond  f)(\omega) = (i\omega)^{-1} (\Fo f)(\omega)$.
\end{lemma}

 \begin{proof}
 As noted  above, $\Io_\diamond  f \in L^2(\R)$ for any $f \in C_\diamond(\R)$ and therefore $\Io_\diamond$ is well defined and clearly linear. In order to show that  $\Io_\diamond$ is densely defined  it is sufficient to show that $C_\diamond(\R)$  is dense  with respect to the $L^2$-norm in $C_c(\R)$. For that purpose, let  $g \in C_c(\R)$ and assume without loss of generality that  $\supp (g) \subseteq [-1,0]$. For any $n \in \N$ define    $ g_n(x)  = g(x)$  for $x < 0$, $ g_n(x)  = - g(-x/n)/n$ for $x \in [0,n]$ and
$g_n(x) =  0$ otherwise. Then  $g_n \in C_\diamond(\R)$ and $\norm{g-g_n}^2 \leq \norm{g}^2/n \to 0$.  Hence $C_\diamond(\R)$ is dense in $C_c(\R)$. Furthermore $\signal, \Io_\diamond  \signal$ are integrable and therefore $\Fo f, \Fo\Io_\diamond  $ are continuous  functions defined by the standard Fourier integral. Integration by  parts shows  $\Fo \Io_\diamond \signal (\om) = \int_\R e^{-ix \omega}  \Io_\diamond f(x) \diff x = (i\omega)^{-1} \int_\R e^{-ix\omega}  f(x) \diff x = (i\omega)^{-1} \Fo \signal (\om)$.  Because $\Fo$ is an isomorphism,  $C_\diamond(\R) \subseteq L^2(\R)$ is dense, and $\omega \mapsto \omega^{-1}$ is unbounded, which shows  the unboundedness of  integration operator $\Io_\diamond$.
\end{proof}

Below we  extend  $\Io_\diamond $  to a closed  operator on $L^2(\R)$ with dense but non-closed domain.  For that purpose we  recall some basic facts on the adjoint and closure of unbounded operators that we will use for our purpose.

\begin{remark}[Adjoint and closure of unbounded operators]\label{rem:adjoint}
For a densely defined potentially unbounded operator $\Ko \colon \dom (\Ko) \subseteq \X \to \Y$ one defines the adjoint domain $\dom (\Ko^*)$ as the set of all  $\data \in \Y$  such that $\inner{\Ko (\cdot) , \data}  \colon \signal  \mapsto \inner{\Ko \signal , \data}$ is bounded.   According    to the Riesz representation theorem  for any $\data \in  \dom (\Ko^*) $ there exists  a unique element $\Ko^*  \data \in \X$ such  $ \inner{\Ko (\cdot) , \data}  = \inner{ \cdot , \Ko^*  \data}$, which defines the  adjoint $\Ko^* \colon \dom (\Ko^*) \subseteq \Y \to \X$. A linear operator   is called closable if it has an extension to a closed linear operator. It is known that $\Ko$ is closable if and only if $\dom(\Ko^*) $ is dense,  in which case $ \Ko^{**}$ is the  closure of $\Ko$.
\end{remark}

Following  Remark~\ref{rem:adjoint}, we next extend $\Io_\diamond$  to a closed operator $\Io \colon \dom(\Io) \subseteq L^2(\R) \to L^2(\R)$   and further summarize some  basic properties  that will be of later use. We call this extension the  integration operator on $L^2(\R)$.

\begin{proposition}[Integration operator on $L^2(\R)$]\label{prop:I} \mbox{}
\begin{enumerate}[label=(\alph*), topsep=0em]
\item\label{prop:I1}
$\Io_\diamond$  has a closed extension $\Io \colon \dom(\Io) \subseteq L^2(\R) \to L^2(\R)$.

\item\label{prop:I2}
 $\dom(\Io) = \dom(\Io^*) = \{f \in L^2(\R) \mid   (i\om)^{-1} \Fo f(\om) \in L^2(\R) \}$.

\item\label{prop:I3}
$ \forall f \in \dom(\Io) \colon \Fo \Io f (\om) = (i\om)^{-1} \Fo f(\om)$.
\item\label{prop:I4}
$\Io^* = - \Io$.
\item\label{prop:I5}
 $\Io$ is  injective with dense range $\ran(\Io) = \{g \in L^2(\R) \mid   (i\om) \Fo g(\om) \in L^2(\R) \}$.

\item\label{prop:I6}
$ \forall g \in \ran(\Io) \colon   \Fo \Io^{-1} g(\om) = (i\om) \Fo g(\om)$.
\end{enumerate}
\end{proposition}

\begin{proof}
By Plancherel's theorem and Lemma~\ref{lem:I} we have $\inner{\Io_\diamond f, g} = (2\pi)^{-1} \int_\R (\Fo \signal)   (i\om)^{-1} \overline{\Fo g}$. Hence $\inner{\Io_\diamond (\cdot), g} $  is bounded iff  $(i\om)^{-1}  \Fo g \in L^2(\R^2)$ and $\dom(\Io^*) = \{\data \in L^2(\R) \mid   (i\om)^{-1} \Fo \data(\om) \in L^2(\R) \}$.  In particular $\dom(\Io^*)$ is dense which gives  \ref{prop:I1}.  Similar arguments show $\dom(\Io) = \dom(\Io^*)$ and \ref{prop:I2}-\ref{prop:I4}.  Items \ref{prop:I5}, \ref{prop:I6} are immediate consequences of  \ref{prop:I2}, \ref{prop:I3}.
\end{proof}

Our aim is the stable inversion of the integration operator $\Io$, which is equivalent to the stable evaluation of differentiation. For that purpose we  use the concept of filtered TI-DFDs  introduced in this paper. In particular we use the TI-wavelet transform.

\subsection{TI-DFDs  for the integration operator}
\label{ssec:TIWVD}

If $(u_\la, \Vo^*_\la, \qsv_\la)_{\la \in \La} $ is  a TI-DFD for $\Io$, then the translation invariance of $\Io$ implies the  translation invariance of $\Vo^*_\la$. We will therefore consider in the following the case where $\Vo^*_\la \data = v_\la^* \ast \data$ for a TI-frame $(v_\la)_{\la \in \La}$ of $L^2(\R)$.  We will also refer to $(u_\la, v_\la, \qsv_\la)_{\la \in \La} $ as TI-DFD.

Before constructing a wavelet based TI-DFD we start by necessary conditions to be satisfied by the TI-DFDs.

\begin{proposition}[Necessary condition for TI-DFD]\label{prop:DFDi}
Let  $(u_\la, v_\la, \qsv_\la)_{\la \in \La} $ is  TI-DFD for $\Io$. Then
 $u_\la$  are  weakly differentiable with weak derivative $\partial_x u_\la =  - v_\la / \qsv_\la$ and
\begin{linenomath*}
\begin{equation}\label{eq:vi}
	\forall \la \in \La \colon \quad \hat v_\la(\om)  = \qsv_\la \cdot (-i\om) \cdot \hat u_\la(\om) \,.
\end{equation}
\end{linenomath*}
\end{proposition}

\begin{proof}
According to  \ref{ti3} we have  $v_\la^* \ast (\Io f) = \qsv_\la \cdot (v_\la^* \ast f)$ for all $f \in \dom(\Io)$. With  the  Fourier convolution theorem and the Fourier representation of $\Io$ we get   $(i\om)^{-1} \cdot ( \overline{\Fo v_\la} )\cdot   \Fo f = \qsv_\la \cdot (\overline{\Fo u_\la} ) \cdot \Fo f$. Since $\dom(\Io)$  is dense this gives \eqref{eq:vi}. According to \eqref{eq:vi},  $(i\om) \cdot \hat u_\la$ are  in $L^2(\R^2)$ and  $u_\la$ has weak  derivative  $\partial_x u_\la =  - v_\la / \qsv_\la$.
\end{proof}

Proposition~\eqref{prop:DFDi} states that a necessary condition  for $(u_\la)_{\la\in \La}$   to be part of  TI-DFD  is that   all $u_\la$ are weakly differentiable and that  there  exist constants $\qsv_\la>0$ such that  $(-\qsv_\la \partial_x u_\la)_{\la\in \La}$  again is a TI-frame.  According to  Proposition \ref{prop:ti} this means that there are constants $ A_\diamond, B_\diamond$ such that  $2\pi A_\diamond \leq  \sum_{\la\in \La}  \abs{\qsv_\la \om \hat u_\la(\om)}^2 \leq 2\pi B_\diamond$.    In the following theorem we provide a class of  examples using TI-wavelets $(u_j)_{j \in \Z }$ of the form
\begin{linenomath*}
\begin{align} \label{eq:ti-wavelet}
\forall j \in \Z \colon \quad 	u_j(x)  &=  2^j u(2^j x), \\ \label{eq:ti-vag}
\forall j \in \Z \colon \quad 	v_j(x)  &=   - 2^{-j} \partial_x u_j(x) = - 2^{j} (\partial_x  u) (2^j x) \,.
 \end{align}
\end{linenomath*}
In this case we call the elements $v_j$ TI-vaguelettes and the corresponding TI-DFD $(u_j, v_j, 2^{-j})_{j \in \Z} $ a translation invariant wavelet-vaguelette decomposition (TI-WVD).  Note that system  \eqref{eq:ti-vag} is not automatically a TI frame even if $(u_j)_{j\in \Z}$ is. Roughly spoken it  requires  $\hat u_\la$ to be sufficiently well localized  such that  $\qsv_\la \cdot (-i\om) \hat u_\la $  behaved  in some sense similar to  $\hat u_\la$. A more precise condition will be deduced from the proof of the following theorem.

\begin{theorem}[TI-WVD for $\Io$]\label{thm:DFDii}
Let $(u_j)_{j \in \Z}$ be a TI-wavelet frame of the form \eqref{eq:ti-wavelet} such that $\supp (\widehat u) \subseteq \{\omega \in \R \mid a \leq \abs{\omega} \leq b \}$ for some $a, b>0$. Then with $(v_j)_{j \in \Z}$ as in \eqref{eq:ti-vag},  the family $(u_j, v_j, 2^{-j})_{j \in \Z} $ is a TI-DFD for $\Io$, named  TI-WVD for the integration operator.
 \end{theorem}

\begin{proof}
Property \ref{ti1} is satisfied by assumption. We now show that $(v_j)_{j\in\Z}$ as given in \eqref{eq:ti-vag} is indeed a TI-frame, which according to $\mathcal{V}_j^* \data \coloneqq v_j^* \ast \data$ yields \ref{ti2}.  Equivalently, we have to show $ 2\pi \asymp \sum_{\la\in \La} \abs{(2^{-j}\om ) \cdot  \hat u_j(\om)}^2$. Let  $A,B>0$ be the frame bounds of the $(u_j)_{j \in \Z}$.
Since $\supp (\hat u) \subseteq \{\omega \in \R \mid a \leq \abs{\omega} \leq b\}$, the scaled versions $\hat u_j(\om)$ have support in $\{\omega \in \R \mid a 2^j  \leq \abs{\omega} \leq b 2^j\}$. Therefore, we have $ 2^j a \abs{\hat u_j(\om)}  \leq \abs{\om \hat u_j(\om)}  \leq 2^j b \abs{ \hat u_j(\om)}$. Together  with the frame property  of $(u_j)_{j \in \Z}$ this gives $ 2\pi a^2 A \leq  \sum_{\la\in \La} 2^{-2j}   \abs{\om \Fo u_j(\om)}^2 \leq 2\pi b^2 B$ and concludes the proof. Clearly we have $\Ko^* v_j = \kappa_j u_j$  which gives \ref{ti3}.
\end{proof}

Note that we made the restriction in Theorem \ref{thm:DFDii} to compactly supported wavelets in order to avoid technical difficulties and to focus on the main ideas. Similar results can be derived under weaker assumptions  guaranteeing that $ 2\pi \asymp \sum_{\la\in \La} \abs{(2^{-j} \om) \cdot  \hat u(2^{-j} \om)}^2$. Further note that such results can be extended to derivatives and related operators in higher dimension wich for example, is important for edge detection of images \cite{goeppel2022feature}. Such a derivation is however beyond the scope of this paper.

\subsection{Numerical  realization}
\label{ssec:implementation}

For the numerical simulations presented below we use the Tikhonov-type filter (see Example \ref{ex:tik}). The corresponding  Tikhonov filtered TI-WVD reads
\begin{linenomath*}
\begin{align} \label{eq:tiwvd-diff}
	\reg^{(2)}_\al( \data ) \coloneqq \sum_{j \in \Z}
	w_j \ast  \Bigl(  \frac{2^{-j} }{2^{-2j} + \al} \cdot  (v_j^* \ast \data ) \Bigr)
\end{align}
\end{linenomath*}
The  instability of differentiation is  reflected in the decreasing   quasi-singular values $\qsv_j = 2^{-j}$ with increasing $j$. The filtered TI-WVD stabilizes  the inversion by   replacing the   inverse quasi-singular coefficients $1/2^{-j}$ by the Tikhonov filtered approximations $2^{-j}  / ( 2^{-2j} +\al )$.

The TI-vaguelette coefficients can be computed  by $ v_j^* \ast g  = - 2^{-j} (\partial_x u_j \ast g)  = - 2^{-j} (u_j \ast  \partial_x \data ) $.
The latter form has the  advantage that any existing implementation  for  computing the TI-wavelet coefficients  $u_j \ast \data$ can  be used  for its  evaluation. Such a strategy will be employed in this paper. In summary, the  Tikhonov filtered  TI-WVD reconstruction \eqref{eq:tiwvd-diff} can be implemented by  following  Algorithm.

\begin{algorithm}[Tikhonov-filtered TI-WVD reconstruction]\label{alg:tiwvd}\mbox{}\\
Input: Noisy data  $\data^\delta$.\\
Parameters: Mother wavelets $u$, $w$ and  regularization parameter $\al>0$.\\
Output: Regularized reconstruction  $\signal_\alpha^\delta$.
\begin{enumerate}[label=(A\arabic*), topsep=0em,itemsep=0em]
\item\label{A1} Compute the auxiliary TI-wavelet coefficients $(\Uo_j^* \data^\delta)_{j \in \Z} = (u_j^* \ast \data^\delta)_{j\in\Z} $.
\item\label{A2} Obtain the TI-vaguelette coefficients by $(d_j)_{j \in \Z}  = (- 2^{-j}  \partial_x  \Uo_j^* \data^\delta)_{j \in \Z}$.
\item\label{A3}  Multiplication with Tikhonov-filter coefficients: $(c_j)_{j \in \Z} \coloneqq (2^{-j}  / ( 2^{-2j} + \al^2 ) \cdot d_j)_{j \in \Z}$.
\item\label{alg:syntehsis} Application of TI-wavelet synthesis  $\signal_\al^\delta \coloneqq \Wo((c_j)_{j\in \Z})$.
\end{enumerate}
\end{algorithm}

For the numerical simulation, functions $\signal$, $\data$ and  corresponding  TI-coefficients are given on the interval $[-1,1]$ and discretized   using $N=512$ equidistant sample points. The TI-wavelet operators $\Uo^*$, $\Wo$ are numerically computed  with the PyWavelet package version~{1.1.1} \cite{Lee2019}  in Python~{3.8.8}. The standard decimated Tikhonov-filtered   WVD takes a form similar to \eqref{eq:tiwvd-diff} and reads 
\begin{linenomath*}
\begin{equation}  \label{eq:wvd-diff}
	\wvd_\al( \data ) \coloneqq \sum_{\la \in \La}
\frac{2^{-j} }{2^{-2j} + \al} \cdot   \inner{v_\la,  \data} w_\la \,.
\end{equation}
\end{linenomath*}
Compared to TI WVD the inner and outer convolution are replaced by the downsampled and upsampled convolution, respectively (see Remark~\ref{rem:wvd-conv}) and Algorithm~\ref{alg:tiwvd} is adjusted accordingly.  In particular the realization of \ref{A1} for the TI case  discretizes $u_j^* \ast \data^\delta$ using the maximal number of $N$ samples, whereas in the decimated uses fewer samples for coarser scales.

{\rot The TI-WVD and the decimated WVD  have a similar computational expense, taking only a few milliseconds in practice. Theoretical analysis also reveals a similar  complexity. Both methods require the decimated (or undecimated) wavelet transform and its inverse. Assuming the use of a compactly supported mother wavelet and computing wavelet coefficients for possible all scales, the wavelet transform requires $\mathcal{O}(N)$ FLOPS, while the TI  version requires $\mathcal{O}(N\log(N))$ FLOPS. Hence they only differ by a logarithmic factor.}

\subsection{Comparison methods}

There is wide range  of regularization methods for stable numerical differentiation proposed in the literature \cite{hanke2001inverse,anderssen1974numerical,rice1983smoothing,pereverzev2013recommendedlegendre}. While a detailed discussion on available methods is beyond the scope of this paper, we will compare it to related non TI methods, namely the decimated WVD noted above and  a standard projection method using Legendre polynomials  $
	P_k (x) = (k + 1/2)^{1/2} ( 2^k k !)^{-1} \partial_x^k (x^2 - 1)^k$ taking the form
\begin{linenomath*}
\begin{equation}\label{eq:legendre-diff}
	\legendre_N ( \data ) \coloneqq
	\sum_{k = 1}^N \inner{\data^\delta, P_k} \partial_x P_k \,.
\end{equation}
\end{linenomath*}
In \eqref{eq:legendre-diff} the truncation index $N$ plays the role of the regularization parameter. It is well known $(P_k)_{k\in\N_0}$ forms an ONB of $L^2([-1, 1])$ with respect to the standard inner product  \cite{arfken1999mathematical} and have been used as a recommended  expansion basis for numerical differentiation for  \cite{pereverzev2013recommendedlegendre}. Note that   \eqref{eq:legendre-diff}  follows the spirit to the filtered DFD using the truncation filter. However it does not fit into the DFD framework in general because $\partial_x P_k$ cannot be rescaled to form a Riesz basis. This is also the reason why we do not use the Tikhonov filter for the Legendre expansion. We conjecture that the frame and Legendre type decompositions can be unified in a more general regularization theory, which we will consider as a future research direction.  The Legendre coefficients and the polynomial expansion where calculated using the Numpy package, version 1.20.3.

{\rot Finally, we would like to note that in all numerical simulations, it is necessary to truncate the series for all used methods. However,  for the wavelet  methods we only truncate the series inline with the discretization, meaning that we do not introduce any additional regularization beyond the natural discretization that is inherent in all the methods presented in this paper. While this is an interesting issue, its analysis is beyond the scope of the present manuscript. Opposed to that, in the Legendre  method  the truncation index is actually the regularization parameter.  
}

\subsection{Numerical examples}

In this subsection we  present reconstruction results for the (Tikhonov) filtered TI-WVD  and  compare  them with  the standard decimated  WVD and the truncated Legendre expansion. 
We will demonstrate the feasibility on the three artificial signals. To simulate real applications, we add Gaussian noise to the data  $\data^\delta = \Io \signal + \eta$ where $\eta$ is Gaussian white noise with  standard deviation $\sigma = 0.05$. As underlying wavelet frame we chose the Daubechies wavelet with five vanishing moments. Note that these wavelets form a tight frame and thus the dual frame is given by the  frame itself. We define the corresponding TI vaguelettes by using equation \eqref{eq:ti-vag}.  This  gives an easy implementation of \ref{alg:syntehsis} where we can exchange $\Wo$ by $\Uo$.  {\rot Although Theorem \ref{thm:DFDii} is not applicable in this case, we still expect the results of that theorem to hold. However, we currently do not have a proof for this.}

\noindent\textbf{Linear filtering:}
We first  present  results comparing  \eqref{eq:wvd-diff},  \eqref{eq:legendre-diff} and unregularized finite differences to our TI DFD-based approach \eqref{eq:tiwvd-diff}.
In all cases, the regularization parameters were chosen via a grid search in order to obtain optimal reconstructions and a fair comparison. The parameters $N\in\N_0, \alpha >0$ were optimized in terms of the relative $\ell_2$ error $\norm{\signal_{\rm rec} - \signal}_2/\norm{\signal}_2$ where $\signal$ is the original signal and $\signal_{\text{rec}}$ the reconstruction.

\begin{figure}[h!] \centering
         \includegraphics[height=0.25\textwidth, width=0.31\columnwidth]{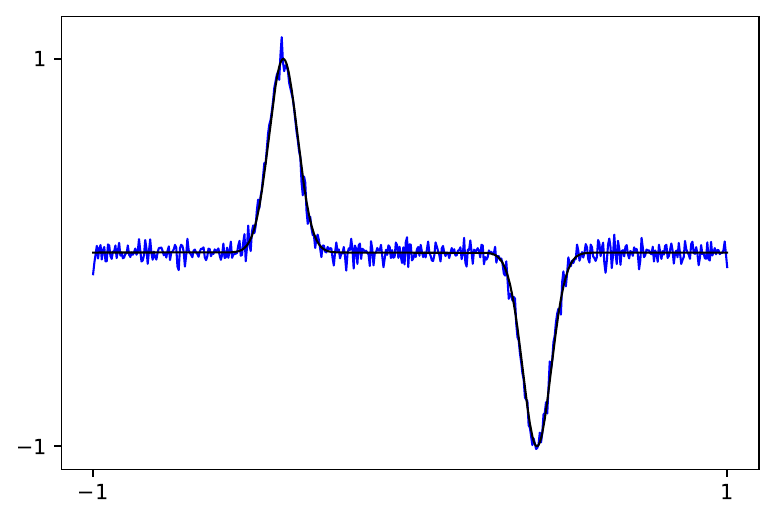}
         \includegraphics[height=0.25\textwidth, width=0.31\columnwidth]{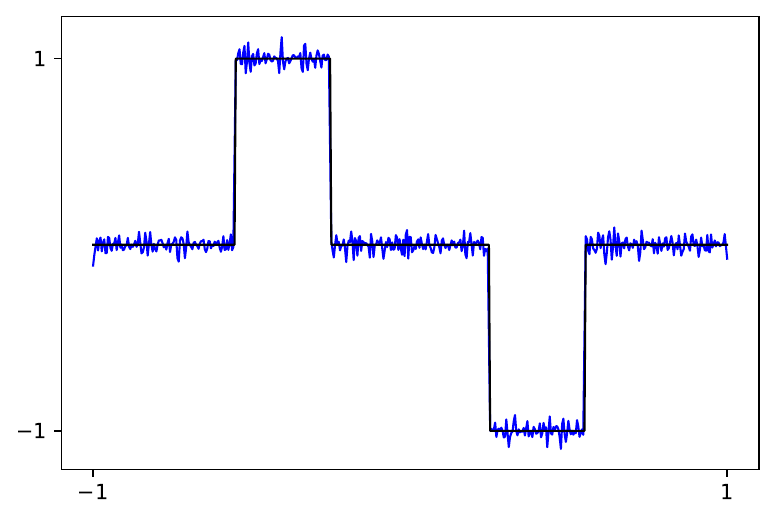}
         \includegraphics[height=0.25\textwidth, width=0.31\columnwidth]{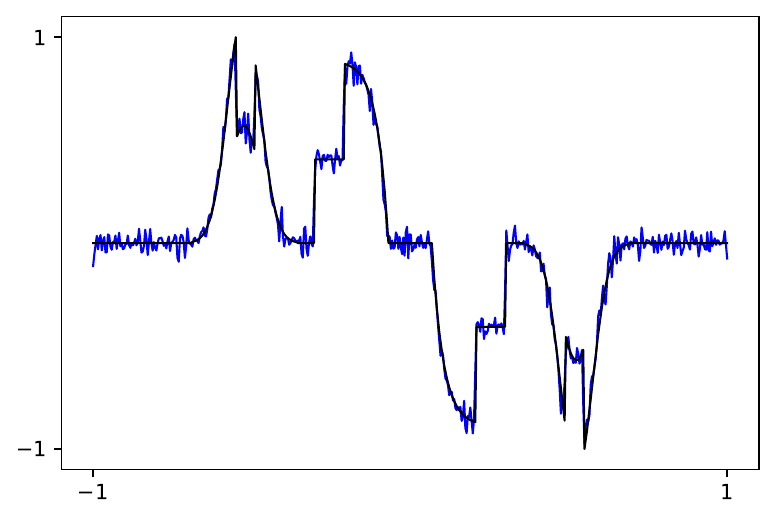}\\
         \includegraphics[height=0.25\textwidth, width=0.31\columnwidth]{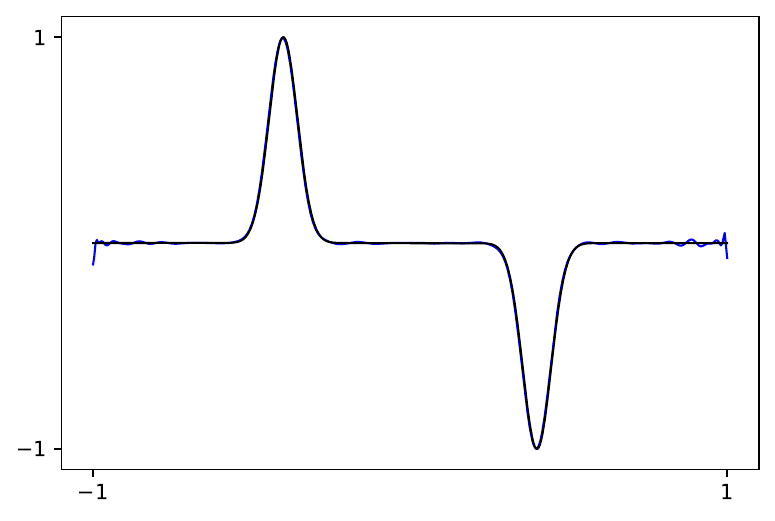}
         \includegraphics[height=0.25\textwidth, width=0.31\columnwidth]{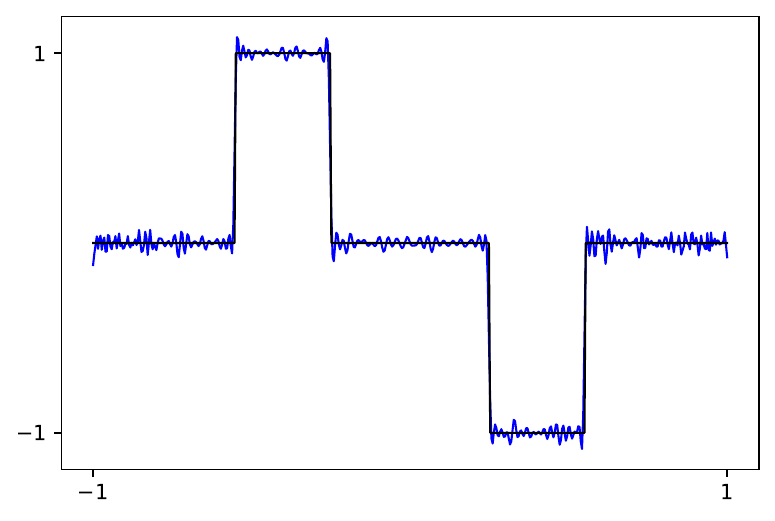}
         \includegraphics[height=0.25\textwidth, width=0.31\columnwidth]{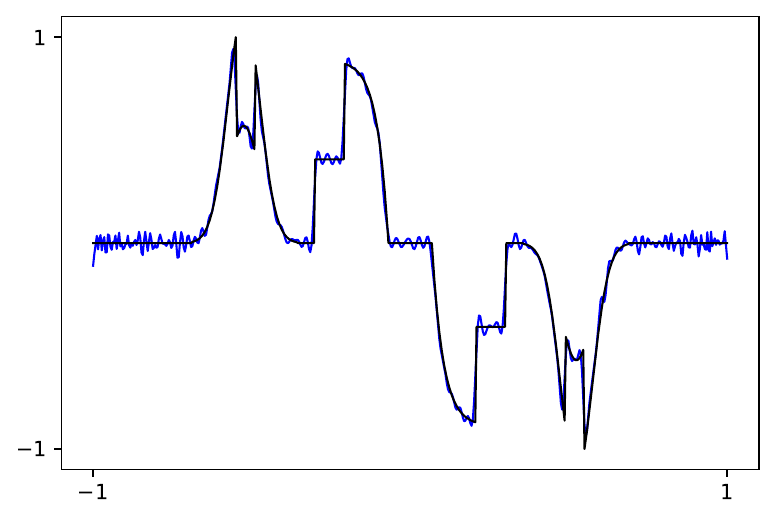}\\
         \includegraphics[height=0.25\textwidth, width=0.31\columnwidth]{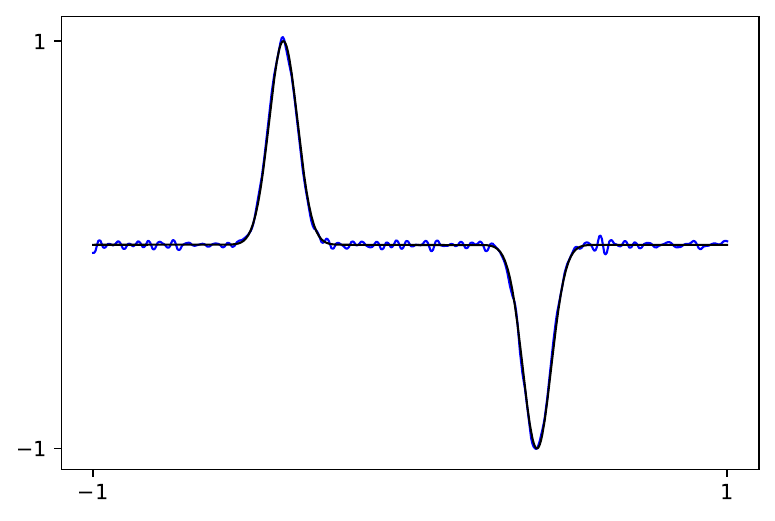}
         \includegraphics[height=0.25\textwidth, width=0.31\columnwidth]{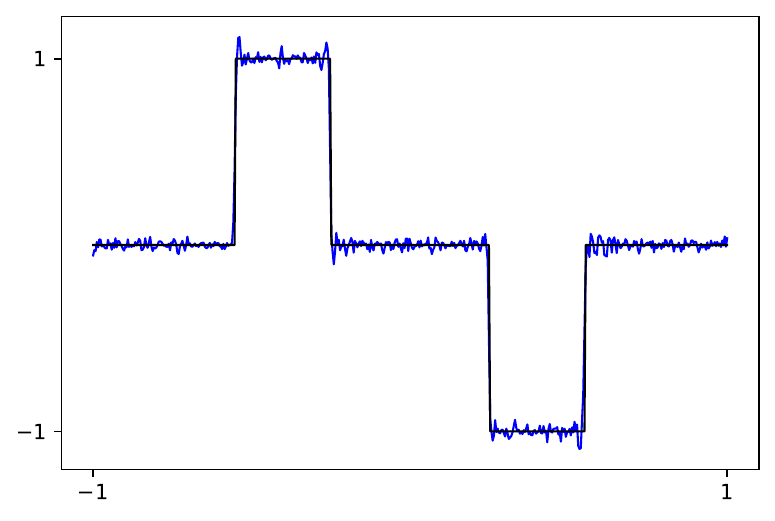}
         \includegraphics[height=0.25\textwidth, width=0.31\columnwidth]{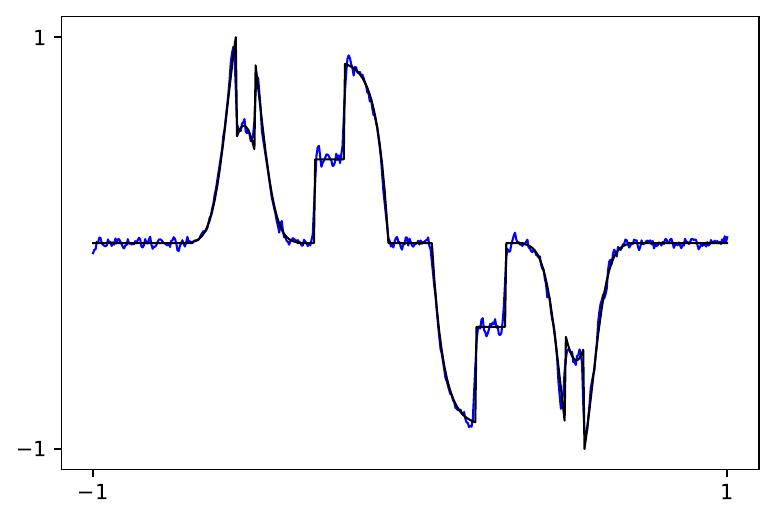}\\
         \includegraphics[height=0.25\textwidth, width=0.31\columnwidth]{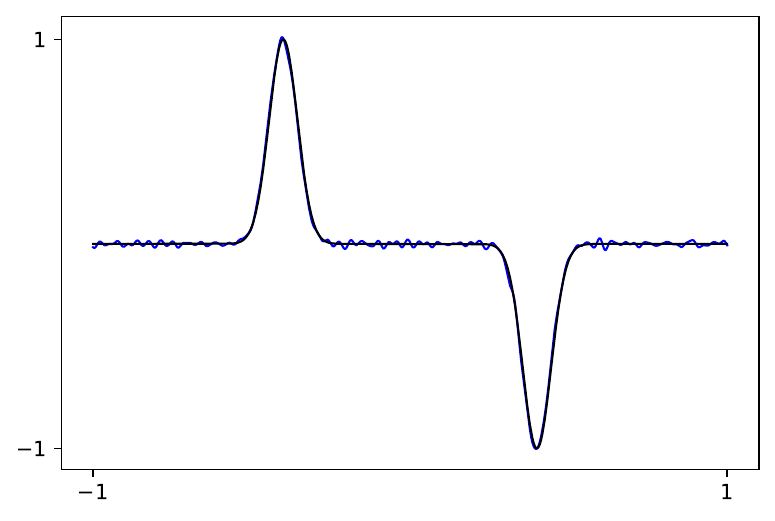}
         \includegraphics[height=0.25\textwidth, width=0.31\columnwidth]{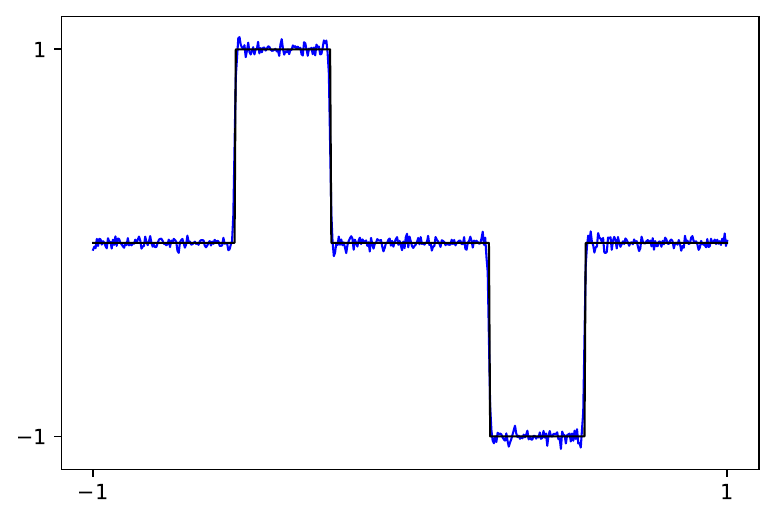}
         \includegraphics[height=0.25\textwidth, width=0.31\columnwidth]{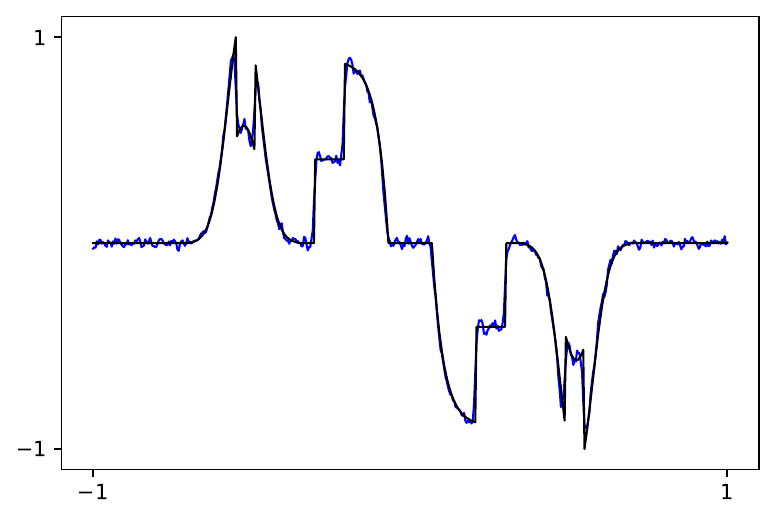}
\caption{Reconstructions from noisy data using finite differences (top row), the Legendre approach (second row 2), WVD (third row) and TI-WVD (bottom row). The original signal is shown in and the corresponding reconstruction in blue.}
\label{fig:noisy}
\end{figure}

The three example functions we will use are given as a smooth signal, a piecewise constant signal and a signal containing sections of both (mixed). Corresponding reconstructions are presented in Figure~\ref{fig:noisy}, where the black graph always shows the original signal and the blue graph represents the obtained reconstruction for any of the reconstruction methods. The first row, shows the non regularized reconstruction obtained by applying finite differences. The second row shows the reconstructions obtained using the truncated Legendre polynomials. The third and fourth row show the filtered WVD and the filtered TI-WVD reconstructions, respectively. We have chosen two decomposition levels

For all example signals, the unregularized reconstruction performs worst. For the smooth  phantom in the first column, the Legendre approach appears to have the best approximation properties. Besides boundary effects, the signal is recovered almost exactly, whereas the wavelet based WVD methods suffer from small wave like artifacts. For the piecewise constant signal and the mixed signal, however, we see that the Legendre approach does not yield a good approximation method anymore. In particular, close to the jumps the approximation quality is poor. In this case  TI-WVD clearly has the best reconstruction quality.  These visual findings are confirmed by quantitative evaluation shown in Table~\ref{tab:errors reg}, which shows the relative $\ell_2$ reconstruction error for any of the reconstructions.

\begin{table}[htb!]
    \begin{center}
    \begin{tabular}{|  l  | l | l | l | }
    \toprule
       & smooth signal & constant signal  & mixed signal \\
       \midrule
   Unregularized & 0.0165 & 0.0103 & 0.0142 \\
       \midrule
   Legendre approach $\legendre_N$ & \textbf{0.0017} & 0.0100 & 0.011 \\
       \midrule
   Decimated WVD $\wvd_\al^{(2)}$ &  0.0029 & 0.0096 & 0.010 \\
       \midrule
   Proposed TI-WVD $\reg_\al^{(2)}$ &  0.0021 & \textbf{0.0089} & \textbf{0.0093} \\
       \bottomrule
    \end{tabular}
    \end{center}
    \caption{Comparison of relative $\ell_2$ reconstruction errors. The best results for each signal are highlighted in boldface.}
    \label{tab:errors reg}
\end{table}

\noindent\textbf{Non-linear filtering:}

As a second example we want to compare the TI-WVD with the standard decimated WVD more closely. For that purpose we  include coefficient thresholding which is  known to optimally remove Gaussian white noise \cite{candes2002,donoho1995nonlinear,haltmeier2014extreme}.
Since we are dealing with multi-scale decompositions we choose level dependent thresholds $t = 2^{-j} \beta$, for some fixed $\beta >0$  and replace the vaguelettes coefficients by $\soft (2^{-j} \beta, v_j^\ast \ast g)$ with soft-thresholding function $\soft(t,x) \coloneqq \sign (x) \max \{0, \abs{x} - t\}$. We have chosen four decomposition levels and again performed a parameter sweep for $\beta >0$ to obtain optimal reconstructions.

Figure~\ref{fig:soft} shows reconstruction from noisy measurements using thresholded TI-WVD and thresholded WVD. The results clearly suggest, that the TI-WVD  yields more reliable reconstruction as the reconstruction is less perturbed by remaining artifacts. Quantitatively, the $\ell_2$ reconstruction error is given by $0.0095$ for the TI-WVD, and $0.0011$ for the decimated WVD. This is inline with reported results for the simple denoising task. {\rot Finally, we would like to note that the TI-WVD reconstruction shows  reduced wavelet artifacts in comparison to its decimated counterpart. Specifically, the oscillating, well-localized errors are significantly reduced in the TI-WVD. }

\begin{figure}[h!]
\centering
\includegraphics[height=0.35\textwidth, width=0.36\columnwidth]{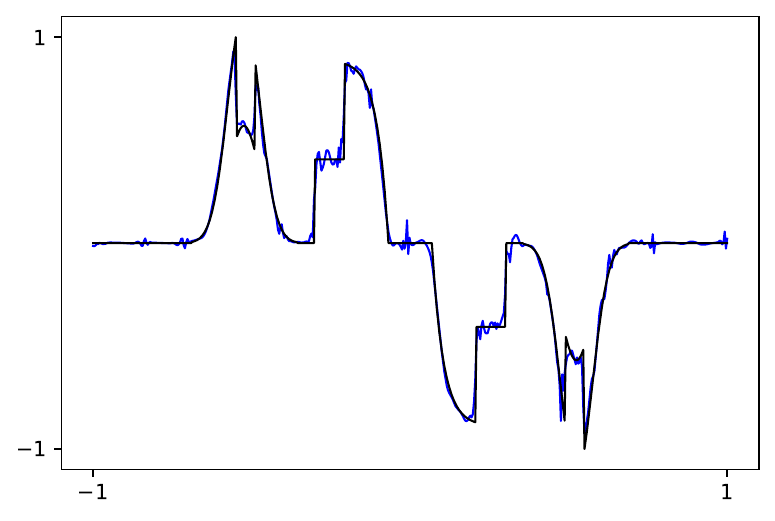}
\includegraphics[height=0.35\textwidth, width=0.36\columnwidth]{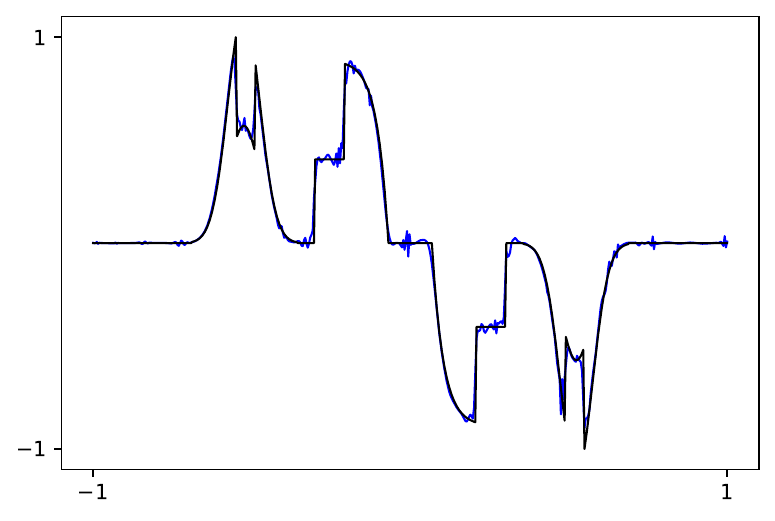}
\caption{\rot Nonlinearly filtered WVD (left) thresholded  TI-WVD (right), where  the filter is  the well established  soft-thresholding function. }
\label{fig:soft}
\end{figure}

\section{Conclusion}
\label{sec:conclusion}

This work presents a new approach for addressing ill-posed inverse problems, called the translation invariant frame decomposition (TI-DFD). We  showed that filtered TI-DFDs yields a regularization method with order optimal rates.  Unlike iterative  and variational methods, the filtered TI-DFD has an explicit form, which enables  efficient implementation.  The translation invariant structure of TI-DFDs has been found to improve reconstruction quality and reduce artifacts in the context of wavelet thresholding.  To demonstrate the effectiveness of TI-DFDs, we constructed a 1D integration example using wavelet frames (translation invariant wavelet vaguelette decomposition; TI-WVD). Our results demonstrate  that filtered TI-WVD outperformed the standard WVD method, with significantly reduced wavelet artifacts. In future research, one promising direction  is exploring the use of nonlinear filters in TI-DFDs or different  parameter selection. Additionally, constructing TI-DFDs for other operators, such as the Radon transform or related transforms using curvelet or shearlet systems ate interesting lines of future research .

\section{Acknowledgement}
The contribution by S.G. is part of a project that has received funding from the European Union’s Horizon 2020 research and innovation programme under the Marie Sk\l{}odowska-Curie grant agreement No 847476. The views and opinions expressed herein do not necessarily reflect those of the European Commission.

\end{document}